\documentclass{amsart}
\usepackage[dvipsnames]{xcolor}
\usepackage[breaklinks,colorlinks,citecolor=blue]{hyperref}
\usepackage{soul}
\usepackage{amsfonts,amscd,amssymb,amsmath,amsthm,mathtools}
\usepackage{graphicx,caption,subcaption,mathrsfs,appendix,cancel}
\usepackage{xparse,enumerate}
\usepackage{tikz}
\usepackage{pgfplots}

\usepgfplotslibrary{external}
\usetikzlibrary{decorations.pathmorphing}

\usepackage[foot]{amsaddr}

\usepackage{verbatim}   %
\usepackage{lipsum}

\newif\ifshowtikz
\showtikzfalse   %

\let\oldtikzpicture\tikzpicture
\let\oldendtikzpicture\endtikzpicture

\renewenvironment{tikzpicture}{%
    \ifshowtikz\expandafter\oldtikzpicture%
    \else\comment%
    \fi
}{%
    \ifshowtikz\oldendtikzpicture%
    \else\endcomment%
    \fi
}

\usepackage[style=alphabetic, maxbibnames=6, maxalphanames=5, doi=true, url=false, isbn=false]{biblatex}
\begin{document}

\newtheorem{theorem}{Theorem}[section]
\newtheorem{lemma}[theorem]{Lemma}
\newtheorem{corollary}[theorem]{Corollary}
\newtheorem{conjecture}[theorem]{Conjecture}
\newtheorem{cor}[theorem]{Corollary}
\newtheorem{proposition}[theorem]{Proposition}
\newtheorem{assumption}[theorem]{Assumption}
\theoremstyle{definition}
\newtheorem{definition}[theorem]{Definition}
\newtheorem{example}[theorem]{Example}
\newtheorem{claim}[theorem]{Claim}
\newtheorem{remark}[theorem]{Remark}

\newenvironment{pfofthm}[1]
{\par\vskip2\parsep\noindent{\sc Proof of\ #1. }}{{\hfill
$\Box$}
\par\vskip2\parsep}
\newenvironment{pfoflem}[1]
{\par\vskip2\parsep\noindent{\sc Proof of Lemma\ #1. }}{{\hfill
$\Box$}
\par\vskip2\parsep}

\newcommand{\red}{\color{red}}
\newcommand{\dist}{\mathrm{d}}
\newcommand{\p}{\mathbf{p}}
\newcommand{\q}{\mathbf{q}}
\renewcommand{\u}{\mathbf{u}}
\newcommand{\R}{\mathbb{R}}
\newcommand{\I}{\operatorname{I}}
\newcommand{\C}{\mathbb{C}}
\newcommand{\UH}{\mathbb{H}}
\newcommand{\T}{\mathbb{T}}
\newcommand{\D}{\mathbb{D}}
\newcommand{\G}{\mathcal{G}}
\newcommand{\V}{\mathrm{V}}
\newcommand{\Z}{\mathbb{Z}}
\newcommand{\Q}{\mathbb{Q}}
\newcommand{\E}{\mathbb E}
\newcommand{\N}{\mathbb N}

\newcommand{\Def}{\overset{\Delta}{=}}

\newcommand{\supp}{\operatorname{supp}}
\newcommand{\sgn}{\operatorname{sgn}}

\newcommand{\Prob}{\Pr}
\newcommand{\Var}{\operatorname{Var}}
\newcommand{\Cov}{\operatorname{Cov}}
\newcommand{\dtv}{d_{\operatorname{TV}}}
\newcommand{\Exp}{\mathbb{E}}
\newcommand{\expect}{\mathbb{E}}
\newcommand{\1}{\mathbf{1}}
\newcommand{\prob}{\Pr}
\newcommand{\weakto}{\Rightarrow}
\newcommand{\probto}{\overset{\Pr}{\longrightarrow}}
\newcommand{\pr}{\Pr}
\newcommand{\filt}{\mathscr{F}}
\DeclareDocumentCommand \one { o }
{%
  \IfNoValueTF {#1}
  {\mathbf{1}  }
  {\mathbf{1}\left\{ {#1} \right\} }%
}
\newcommand{\Bernoulli}{\operatorname{Bernoulli}}
\newcommand{\Binomial}{\operatorname{Binom}}
\newcommand{\Binom}{\Binomial}
\newcommand{\Poisson}{\operatorname{Poisson}}
\newcommand{\Exponential}{\operatorname{Exp}}

\newcommand{\Ai}{\operatorname{Ai}}
\newcommand{\tr}{\operatorname{tr}}
\renewcommand{\det}{\operatorname{det}}
\newcommand{\diag}{\operatorname{diag}}

\newcommand{\Image}{\operatorname{Image}}
\newcommand{\Span}{\operatorname{Span}}
\newcommand{\Id}{\operatorname{Id}}

\newcommand{\EF}{\mathbf{Q}_{N}}
\newcommand{\JEF}{\mathbf{L}_{N}}
\newcommand{\WEF}{\mathbf{W}_{N}}
\newcommand{\ESD}{\varrho_N}
\newcommand{\LSD}{\varrho}
\newcommand{\Dev}{\vartheta_\beta}
\newcommand{\TGF}{\mathbf{T}}
\newcommand{\ZF}{\mathbf{Z}_N}
\newcommand{\WGF}{\mathbf{W}}
\newcommand{\Scary}{\mathcal{S}}
\newcommand{\HB}{ B_{\mathbb{H}}}

\DeclareDocumentCommand \BL { G{t_0} O{\omega} } { 
  {\mathcal{E}(#2)}
  }
\DeclareDocumentCommand \ELL { G{p,t_0}  O{\theta} } { 
  {\mathcal{E}(#2)}
}
\newcommand{\MFI}{{Y}}

\DeclareDocumentCommand{\FMC}{ g }{
  \IfNoValueTF {#1}
  {
    \mathbb{F}
  }
  {
    \mathbb{F}_{ {#1} }
  }
}

\DeclareDocumentCommand \norm { O{\cdot} } { \left\|{ #1 }\right\| }
\DeclareDocumentCommand \nuclear { O{\cdot} } { \left\|{ #1 }\right\|_{\nu} }
\DeclareDocumentCommand \Hhalf { O{\cdot} } { \left\|{ #1 }\right\|_{H^{1/2}} }
\DeclareDocumentCommand \rawev { O{i} O{N} } { \lambda_{ {#1 }}^{({#2})} }

\newcommand{\GF}{\mathbf{G}}
\newcommand{\LF}{\mathbf{L}}
\newcommand{\QGF}{\mathbf{Q}}
\newcommand{\F}{\mathbf{F}}
\newcommand{\BSA}{\mathscr{B}}
\newcommand{\CUEF}{\mathbf{U}}
\newcommand{\GUEF}{\mathbf{\Lambda}}
\NewDocumentCommand {\BF} { } {\nu}
\newcommand{\WN}{\mathbf{Z}}
\newcommand{\RWN}{\mathbf{R}}
\newcommand{\dH}{d_{\mathbb{H}}}
\newcommand{\dT}{d_{T}}
\newcommand{\ddH}{\tilde{d}_{\mathbb{H}}}
\newcommand{\HM}{\mathfrak{m}}
\DeclareDocumentCommand \ROT { O{\theta} } { 
  {Q}_{ {#1} }
}
\newcommand{\Meso}{\mathscr{E}}
\newcommand{\FI}{{W^p}}

\DeclareDocumentCommand \Ev { G{\BF,p,t_0} o } { 
  \EVENT{\mathcal{A}}{\ell}[#1][#2]
}
\DeclareDocumentCommand \EvU { G{\rho,t} o } { 
  \EVENT{\mathcal{A}}{u}[#1][#2]
}
\NewDocumentCommand {\EvL}{ O{\BF} O{\theta} } { 
  \EVENT{\mathcal{A}}{\ell2}[#1][#2]
}
\DeclareDocumentCommand \EvUU { o o } { 
  \EVENT{\mathcal{A}}{u2}[#1][#2]
}
\DeclareDocumentCommand \BS { o o } { 
  \EVENT{\mathcal{B}}{1}[#1][#2]
}

\DeclareDocumentCommand \BLL { o o } { 
  \EVENT{\mathcal{B}}{3}[#1][#2]
}
\DeclareDocumentCommand \ES { G{t_0} o } { 
  \EVENT{\mathcal{E}}{1}[#1][#2]
}
\DeclareDocumentCommand \ESU { o } { 
  \EVENT{\mathcal{E}}{2}[][#1]
}
\DeclareDocumentCommand \EL { G{\rho,t} o } { 
  \EVENT{\mathcal{E}}{u}[#1][#2]
}

\DeclareDocumentCommand \ELM { o o } { 
  \EVENT{\mathcal{E}}{\HM}[#1][#2]
}
\DeclareDocumentCommand \DL { o O{k} } { 
  \EVENT{\mathcal{D}}{#2}[\rho][#1]
}
\DeclareDocumentCommand{\EVENT}{ m m o o }
{
  \IfNoValueTF {#3}
  {
    {#1}_{#2}^{{\BF}}
  }
  {
    \IfNoValueTF {#4}
    { {#1}_{#2}^{{#3}} }
    { {#1}_{#2}^{{#3}}({#4}) }
  }
}

\DeclareDocumentCommand \Bias { g o }{
  \IfNoValueTF {#1}
  {
    \IfNoValueTF {#2}
    {
      \mathfrak{B}
    }
    {
      \mathfrak{B}_{#2}
    }
  }
  {
    \IfNoValueTF {#2}
    {
      \mathfrak{B}(#1)
    }
    {
      \mathfrak{B}_{#2}(#1)
    }
  }
}
\DeclareDocumentCommand \WUSA { O{\ell} }{
  \operatorname{MEM}(#1)
}
\title[RMT field]{The law of large numbers for the maximum of almost Gaussian log-correlated fields coming from random matrices}
\author{Gaultier Lambert}
\address{Department of Mathematics, Universit\"at Z\"urich}
\email{gaultier.lambert@math.uzh.ch}
\author{Elliot Paquette}
\address{Department of Mathematics, The Ohio State University}
\email{paquette.30@osu.edu}
\thanks{ G.L.\,gratefully acknowledges the support of the grant
KAW 2010.0063 from the Knut and Alice Wallenberg Foundation while at KTH, Royal Institute of Technology.  E.P.\,gratefully acknowledges the support of NSF Postdoctoral Fellowship DMS-1304057.
  This work began while G.L.\,visited Weizmann, supported in part by a grant from the Israel Science Foundation.  E.P.\,would also like to thank Kurt Johansson for his invitation to visit KTH, during which time part of this work was completed.
  This project has also received funding from the European Research Council (ERC) under the European Union’s Horizon 2020 research and innovation programme (grant agreement No. 692452).
}
\date{\today}
\maketitle

\begin{abstract}\normalsize
We compute the leading asymptotics as $N\to\infty$ of the maximum of the field 
$Q_N(q)=  \log\det|q- A_N|$, $q\in\C$, for any unitarily invariant Hermitian random matrix $A_N$ associated to a non-critical real-analytic potential. Hence, we verify the leading order in a conjecture of \cite{FyodorovSimmGUE} formulated for the GUE. The method relies on a classical upper-bound and a more sophisticated lower-bound based on a variant of the second-moment method which exploits the hyperbolic branching structure of the field $Q_N(q)$, $q\in\UH$. Specifically, we compare $Q_N$ to an idealized Gaussian field by means of exponential moments. In principle, this method could also be applied to  random fields coming from other point processes provided that one can compute certain mixed exponential moments. For unitarily invariant ensembles, we show that these assumptions follow from the Fyodorov-Strahov formula \cite{FS} and asymptotics of orthogonal polynomials derived in \cite{DKMVZ}.
\end{abstract}

\section{Introduction}
\label{sec:general}

We consider the following general problem, applicable to the study of the maximum of the log-modulus of the determinant of a random matrix with real eigenvalues.  Suppose that $\ESD$ is the empirical measure of a random collection of $N$ real points $\left\{ \lambda_i \right\}_1^N,$ i.e.
\[
  \ESD(\lambda) = \sum_{i=1}^N \delta(\lambda - \lambda_i).
\]
We will work under the assumption that $N^{-1}\ESD$ is converging to a compactly supported, deterministic probability measure with compact density, which we denote by $\LSD.$  Without loss of generality, we will assume the support of $\LSD$ is contained in $[-1,1].$  We will not assume, however, that the limiting measure has a connected support.
In terms of these objects, we define $\EF : \UH_{\pm} \to \R$ as the  log-potential of the measure $N\LSD - \ESD,$ that is for $q \in \UH_{\pm}$,
\begin{equation*}
  \EF(q) =
  \int_\R \log|q-x|\,\ESD(dx)
  -
  N\int_\R \log|q-x|\,\LSD(dx).
\end{equation*}

As is known for many classes of random matrices $A_N$, the process $\EF(q)$ coming from the eigenvalues of  $A_N$ satisfies a central limit theorem for fixed $q \in \UH_{\pm}$ (see e.g.\,\cite{Kurt98,PasturShcherbina,BaiSilverstein}), as $N\to\infty.$ This gives rise to an explicit, centered Gaussian field $\Lambda : \UH_{\pm} \to \R.$ For example, as a consequence of \cite{Kurt98}, for a wide class of one-cut unitarily invariant random matrices, 
\begin{equation} \label{eq:Lambda}
  \EF(q) \weakto_{N\to\infty} \Lambda(q) .
\end{equation}
The one-cut assumption, meaning that the support of the equilibrium density $\LSD$ is connected, is necessary to get the limiting Gaussian behavior.  A similar statement holds for general $\beta$-ensembles as well,  but the limiting field is not centered if $\beta \neq 2$.  For multi-cut ensembles, the limiting distribution is no longer Gaussian \cite[Theorem 2]{Shcherbina} (see also \cite{BorotGuionnet}).

This field $\Lambda(q)$ is a natural example of a log-correlated Gaussian field, many of whose properties are well understood: in particular, there is work on the geometry of thick points \cite{Daviaud, HuMillerPeres} in specific cases (which should be expected to generalize naturally) and work on the law of the maximum in great generality \cite{DRZ}.  There is also work on the convergence of exponentials of non-Gaussian log-correlated fields and their convergence to Gaussian multiplicative chaoses \cite{SaksmanWebb, Webb}.  This article is philosophically concerned with determining to what extent predictions about the maximum of such a Gaussian log-correlated field also hold for $\EF.$

Along this line, \cite{FyodorovSimmGUE} (see also \cite{FH}) have made a prediction for the maximum of the log-determinant of a Gaussian Unitary Ensemble (GUE) matrix, based on a hypothetical analytic continuation of the Selberg integral. 
\begin{conjecture}
  \label{conj:FyodorovSimm}
  Suppose $A_N$ is the Gaussian Unitary Ensemble. Let
  \[
    M_N^* = \max_{x \in [-1,1]} \left\{ 
    \log|\det(x-A_N)| - \Exp(\log |\det(x - A_N)|)
    \right\}.
  \]
  Then as $N\to\infty$ 
  \[
    M_N^* - \log N + \frac34 \log\log N \weakto y,
  \]
  where $y$ has an explicit distribution (see \cite{FyodorovSimmGUE} for more details).
\end{conjecture}
In effect, in line with what is seen for log-correlated Gaussian fields, the maximum grows like $\log N - \frac 34 \log\log N$ plus a fluctuating term.  This fluctuating term depends strongly on the details of the model, although some features of it are universal.

There is a parallel story to this which has been much more fully developed for the log-potential of the CUE and its $\beta$-analogue the C$\beta$E.  There, the leading order term was first proven to hold by \cite{ABB} for the CUE, the second order term by \cite{PaquetteZeitouni} and most recently the tightness of the recentered maximum for the C$\beta$E by \cite{CMN}.

\begin{remark}
It is also possible to consider the imaginary part of the characteristic polynomial, defined as
\[
  \Im \log p_N(x) = \sum_{i=1}^N \Im \log(x-\lambda_i),
\]
where for each summand we take the principal branch of the logarithm.  %
The same predictions as for the real part of the logarithm should hold for $\Im \log p_N(x),$ appropriately adapted.  We do not consider this however, as part of our method is limited to the real part of the logarithm (specifically our reliance on the Fyodorov-Strahov formula, \eqref{Fyodorov_Strahov}).
\end{remark}

\subsection*{Our results}

We show that the first term in the conjectured expansion \ref{conj:FyodorovSimm} holds.  
Moreover, we show that this holds uniformly over a large class of random matrix ensembles with analytic potentials on $\R$. 
We will work under very few assumptions on the potential $V$ (see \cite{DKMVZ} for background).  Specifically, we assume the potential is real-analytic and \emph{regular}.  Under the first assumption, the measure $\LSD$ is supported on finitely many intervals and it has a bounded density. The second assumption implies that the density of  $\LSD$ vanishes like a square-root at the edges of these intervals and is strictly positive in the interior of the support.   

\begin{theorem}
  \label{thm:main}
  Let $V$ be a regular, real analytic potential.  Then $M_N^*/\log N \to 1$ in probability as $N\to \infty.$ %
  The lower bound holds without the requirement that $V$ is regular.
\end{theorem}

We go to great lengths to show that, in some sense, the needed random matrix asymptotics already exist in the literature. 
The hypotheses on the class of random matrices involved is effectively only limited by the availability of orthogonal polynomial asymptotics. 
In our setting, we rely on \cite{DKMVZ}, which establishes the orthogonal polynomial asymptotics for varying weights on the real line by the Riemann-Hilbert steepest descent method.  
Further, in some sense, the true assumptions needed are substantially weaker than these full asymptotics.  (Moreover, we only assume that $V$ is regular to get a less technical proof of the upper-bound.  We believe that it is not necessary and could be removed using the asymptotics of  \cite{DKMVZ} near the singular points.)

\subsection*{General theory}
To prove Theorem~\ref{thm:main}, we develop an abstract machinery suited for controlling  these almost-Gaussian log-correlated fields, which in principle could be applied outside random matrix theory.  

Theorem~\ref{thm:main} consists of an upper bound and a lower bound.  The upper bound, in a sense, is much easier.  In effect, using that $\EF$ is harmonic off the real line and arises from an $N$-point measure, the problem of controlling $\EF$ from above can be reduced to estimating the Laplace transform $\Exp e^{2\EF(q)}$ for $q$ near the real line.  See Section~\ref{sec:ub} for further details.

The more complicated task is to develop a lower bound for $\EF.$  This is where we nontrivially use the limiting log-correlated structure of the field.
We begin by introducing a Gaussian harmonic function on the unit disk.
Let $\GF$ be the centered Gaussian with covariance
\begin{equation}\label{covariance_GF} 
\Exp \left[\GF(z)\GF(w) \right] = -\frac12 \log|1-z\overline{w}|.
\end{equation}
This appears as the limiting field for the log-determinant of a Haar-unitary matrix, see \cite{PaquetteZeitouni}.
The field $\GF$ is conformally invariant.  Specifically, recall the hyperbolic disk automorphism that for any $y \in \D$ is given by the map
\begin{align}
    T_y : \D \to \D, z \mapsto \frac{z-y}{1-z\bar{y}},
  \label{eq:diskautomorphism}
\end{align}
which is an isometry of the Poincar\'e disk taking $y$ to $0.$  
Then for any $y \in \D,$ $z \mapsto \GF(T_y(z)) - \GF(y)$ has the same distribution as $\GF.$  This leads its structure to naturally be described in terms of hyperbolic geometry.  
Let $\dH$ be the hyperbolic metric on $\mathbb{D}.$  For any point $z \in \mathbb{D},$ the distance of $z$ to $0$ this can be given by 
\begin{align}
  \dH(0,z) = \log \left( \frac{1+|z|}{1-|z|} \right). \nonumber
\end{align}
For two arbitrary points $y,z \in \mathbb{D},$ we can then write
  \( \dH(y,z) = \dH(0,T_y(z)). \)
A short calculation shows that the covariance structure of $\GF$ can alternatively be expressed by
\begin{align}
  \Exp[ \GF(z)\GF(y) ]
  &= \frac12\log\left( 
  \tfrac
  { \cosh( {\dH(0,y)}{2}^{-1} )\cosh( {\dH(z,0)}{2}^{-1} )}
  {\cosh( {\dH(z,y)}{2}^{-1} )}
  \right). \label{eq:covcosh}
\end{align}

One advantage of this expression is that the function $x \mapsto \log(\cosh(x))$ is uniformly Lipschitz.  Hence, for example, the correlation of an increment with any point in the field can be controlled solely in terms of the length of the increment:
\begin{equation}
  \label{eq:correlationbound1}
  \sup_{z,y,x \in \D} \frac{|\Exp\left[ (\GF(z)-\GF(y))\GF(x) \right]|}{\dH(z,y)} < \infty.
\end{equation}

The hyperbolic nature of the field $\GF$ leads it to have a natural connection to branching random walk.  Let $\left\{ \zeta_i \right\}_0^\infty$ be points on the positive real axis with $\zeta_0 = 0$ and $\dH(\zeta_i,\zeta_j) = |i-j|.$   
For $\theta \in \R,$ we wish to estimate the distance $\dH(\zeta_i, e^{i\theta}\zeta_j).$  The following Lemma, taken from \cite{PaquetteZeitouni}, exposes the branching structure of the distances and the consequential branching random walk comparison that is possible for the covariances of $\GF$. 
\begin{lemma}
  Uniformly in $h,j \in \mathbb{N}$ and $\theta \in [-\pi,\pi]$
  \[
    \dH(\zeta_h, e^{i\theta}\zeta_j)
    = h + j - 2\min\{-\log|\sin\tfrac{\theta}{2}|,h,j\} + O(1).
  \]
  When $k = \min\{h,j\} > -\log |\tfrac{\sin\theta}{2}|$ the error term can be estimated by $Ce^{-k}|\theta|^{-1}$ for some sufficiently large absolute constant $C>0.$ For the covariances of $\GF,$ it follows that
  \[
    \Exp \GF(\zeta_h)\GF(e^{i\theta}\zeta_j)
    =\tfrac{1}{2}\min\{-\log|\sin\tfrac{\theta}{2}|,h,j\}
    - \tfrac{\log 2}{2} + O(1),
  \]
  where again the error term can be estimated by $C\min\{e^{-k}\theta^{-1},1\}.$
  \label{lem:branch}
\end{lemma}
\begin{proof}
  For a hyperbolic triangle with side lengths $a,b,c$ with $\theta$ the angle opposite $a,$ the hyperbolic law of cosines says that
  \[
    \cosh a = 
    \frac{\cosh(b+c)}{2}(1-\cos\theta)
    +\frac{\cosh(b-c)}{2}(1+\cos\theta).
  \]
  We apply this with $a=\dH(\zeta_h,e^{i\theta}\zeta_j),$ $b=h$ and $c=j.$
  The remainder is a straightforward case-by-case analysis, noting that when  $k = \min\{h,j\} > -\log |\sin\tfrac\theta2|,$ the first term dominates, and otherwise the second term dominates.  Using \eqref{eq:covcosh}, this estimate can be transferred to the covariances, since for $x \geq 0,$
  \[
      \log(\cosh(\tfrac{x}{2})) = \tfrac{x}{2} - \log 2 + O(e^{-x}).
  \] 
\end{proof}

\subsection*{Connecting $\EF$ to $\GF$}

The point of introducing the field $\GF$ is that $\EF,$ at least in the neighborhood of a point in the bulk of the spectrum, is well approximated in law by $\GF.$  In the case that $V$ is one-cut regular, the connection is particularly strong (because of the CLT \eqref{eq:Lambda}). 
If the equilibrium density $\LSD$ is scaled so that its support is exactly $[-1,1]$ and we pull-back the field $\EF$ in the unit disk using the Joukowsky map  
\begin{equation} \label{eq:Joukowsky}
J(z) = \frac{z + z^{-1}}{2} ,
\end{equation}
 then $\EF \circ J$ will converge as a subharmonic function on $\C \setminus [-1,1]$ in the local uniform topology to a Gaussian harmonic function $\TGF = \Lambda \circ J $ on $\D.$
By an explicit calculation, it can be checked that this field has covariance
\begin{equation}
  \label{eq:tcov}
  \Exp\left[
    \TGF(z)\TGF(w)
  \right]
  =
  -\frac12 \log|1-zw|
  -\frac12 \log|1-z\overline{w}| . 
\end{equation}
This allows $\TGF$ to be alternately expressed as $z \mapsto (\GF(z) + \GF(\bar z))/\sqrt{2}.$  
Moreover,  one can see that in any small neighborhood of the boundary of the disk away from the edge points $1$ or $-1$ (this points are the only fixed point of the Joukowsky transform), the contribution of $-\frac12 \log|1-zw|$ will be uniformly bounded.  In effect, the covariance structure is approximately that of $\GF$ itself.

In the multicut situation, the global Gaussian convergence of $\EF$ is no longer true, but there is still a type of local Gaussian convergence:  it is still possible to compare the fluctuations of $\EF\circ J$ to $\GF$ in the neighborhood of a point in the bulk of the spectrum.
This leads us to consider a class of Gaussian fields that are locally like $\GF.$

\begin{definition}
  \label{def:brwlike}
  Say a centered Gaussian field $\WGF$ is \emph{BRW-like in a set $U \subset \D$} if 
  \begin{enumerate}[(a)]
    \item For any $z \in U,$ the function $w \mapsto \Exp\left[ \WGF(z)\WGF(w) \right]$ is harmonic in $U.$
    \item There is a constant $C>0$ so that for all $z,w \in U,$ with $\dH(z,w) \leq 1,$
      \[
	\Var( \WGF(z) - \WGF(w)) \leq C\dH(z,w)^2.
      \]
    \item There is a constant $C>0$ so that for all $y,z,w \in U,$ with $\dH(z,w) \leq 1,$
      \[
	|\Exp\left[ \WGF(y)( \WGF(z) - \WGF(w)) \right]| \leq C\dH(z,w).
      \]
    \item There is a function continuous function $K : \T^2 \to \R$ so that for all $\zeta_h e^{i\theta_1} \in U$ and $\zeta_j e^{i\theta_2} \in U,$
      \[
	\Exp \WGF(\zeta_h e^{i\theta_1})\WGF(\zeta_j e^{i\theta_2})
	=\tfrac{1}{2}\min\{-\log|\sin(\tfrac{\theta_1-\theta_2}{2})|,h,j\}
	+K(\theta_1,\theta_2) + O(1),
      \]
      where the error term goes to $0$ uniformly as $\min\{h,j\} + \log|\sin\tfrac{\theta_1-\theta_2}{2}| \to \infty.$ %
  \end{enumerate}
\end{definition}

\begin{remark}
The first condition, that the covariance is harmonic, implies that $\WGF$ is almost surely harmonic in $U$ and  it can be easily checked that 
\[
\Exp \biggl(
\int_\gamma \WGF(z) d\mu_w(z) - W(w)
\biggr)^2 = 0,
\]
for smooth, Jordan curves in $U$ with $w$ a point in the interior of the curve and $\mu_w(z)$ the harmonic measure on the curve seen from $w$.
\end{remark}

\subsection*{Lower bound overview}

The field $\WGF,$ by virtue of being Gaussian and having a log-correlated structure, is amenable to general tools for Gaussian fields: for example very precise information on its maximum on any hyperbolic ball is available using theory in \cite{DRZ}.  In the problems we study here, we are given a sequence of \emph{harmonic} random fields $\WEF$ that are not necessarily Gaussian but are well-approximated by $\WGF. $   Our task is to provide simple conditions of comparison that will produce a lower bound for the maximum of $\WEF$ in the large $N$ limit. These conditions are given solely in terms of mixed moments of exponentials of $\WEF.$ 

In what follows, we fix $\delta >0$ a small positive constant, that we will ultimately take to $0.$  We let $n_0 = \lfloor (1-\delta) n \rfloor$ and introduce the set 
$$\Omega = \left\{ e^{i(\tfrac{\pi}{2} + h e^{-n_0})} : h \in \mathbb{Z}, |h| < N^{-\delta}e^{n_0} \right\}.$$  We will find a lower bound for $\WEF$ by considering the values of $\WEF$ at the points $\left\{ \omega \zeta_{n_0} : \omega \in \Omega \right\}.$ 
To do this, however, we will also consider the behavior of $\WEF$ at points closer to the origin of $\D.$
So, define a domain
\begin{equation*}
  \mathcal{D}_{N,\delta}
  =\left\{ 
    i re^{i\theta}
    \ |\ 1 - N^{-\delta} \leq r \leq 1 - N^{-1+\delta},
    |\theta| \leq N^{-\delta}
  \right\}.
\end{equation*}
This is the set in which we compare $\WEF$ and $\WGF.$

We will take $b^*=b^*_N$ to be an integral valued sequence so that $b^*_N \asymp \delta \log N$ (it will be defined precisely in Section~\ref{sec:lb}).  Then, we will define the field statistics, for $\omega \in \Omega,$
\[
  \Bias[\omega]: \F \mapsto
    2\F(\omega \zeta_{n_0})
    -2\F(i \zeta_{b^*}).
\]
It will essentially suffice to show that some $\Bias_\omega(\WEF)$ is on the order of $(1-O(\delta))\log N.$
Define $\mathfrak{W}_{\ell,\delta}(\Bias[\omega])$ as the set of all cylinder functions
\[
 \Bias: \F \mapsto \Bias[\omega](\F) + 
  \sum_{z \in \mathbf{z}} 2\F(z)
  -\sum_{w \in \mathbf{w}} 2\F(w),
\]
where $\mathbf{z}$ and $\mathbf{w}$ are subsets of $\mathcal{D}_{N,\delta}$ with  $|\mathbf{z}|=|\mathbf{w}|=\ell$
and the property that:
if we denote $Z =  \mathbf{z} \cup \left\{ \omega \zeta_{n_0} \right\}$ and $W=\mathbf{w} \cup \{ i \zeta_{b^*}\}$,
    there is a bijection $\phi : Z \to W$  so that for all $z\in \mathbf{z}$
    \begin{equation*} 
      \dH(z,\phi(z))\leq 
       \min\left\{  
      \min_{z' \in Z\setminus \{z\}} 
      \dH(z,z') ;
      \min_{w \in W \setminus \{\phi(z)\}} 
      \dH(w,\phi(z)) \right\}.
    \end{equation*}
    These field statistics are in some sense local perturbations of $ \Bias[\omega](\F).$

    We make a quantitative assumption on how well $\Bias{\WEF}$ can be approximated by $\Bias{\WGF}$ in the sense of exponential moments.  
    \begin{assumption}
      \label{ass:wusa}
      Let $\ell \in \N.$  We define the \emph{mixed exponential moment} assumption $\WUSA[\ell]$ to be that the following holds.
      For all $\delta > 0$ sufficiently small:
      \begin{enumerate}
	\item The function $\WEF(z)$ is almost surely harmonic in $\mathcal{D}_{N,\delta}.$ 
	\item Uniformly in $\omega \in \Omega$ and $\Bias \in \mathfrak{W}_{\ell,\delta}(\Bias_\omega),$
	  \[
	    \Exp\left[ e^{\Bias{\WEF}} \right]
	    =
	    \Exp\left[ e^{\Bias{\WGF}} \right]
	    (1+O(N^{-\delta})).
	  \]
      \end{enumerate}
    \end{assumption}

    At first sight, it may not be clear that these assumptions even imply, in any sense, that $\WEF$ converges in law to $\WGF.$ However, as a corollary of these assumptions, the mixed moments of $\WEF$ can be compared to the mixed moments of $\WGF$ with a vanishing error (see Proposition~\ref{prop:mixedmoments}).  
    Under the assumption $\cap_{\ell=1}^\infty \WUSA[\ell],$ which we show holds for $\EF \circ J,$ one can deduce Gaussian scaling limits of $\WEF$ as a direct consequence of Proposition~\ref{prop:mixedmoments}.

\begin{theorem}
    \label{thm:subharmoniclb}
    Under Assumption $\WUSA[2]$, for any $\delta >0,$
    \[
      \limsup_{N\to\infty} \Pr[\max_{z \in \D} [\WEF(z) - \WEF(i\zeta_{b*})] < (1-\delta){\log N}]=0.
    \]
\end{theorem}
\noindent One can make a comparison between this theorem and four-moment theorems of \cite{TaoVu} -- a small number of moments of the exponential of the field determines the maximum -- though the methods of proof could not be more different.

We give the proof of this theorem and an overview of the method in Section~\ref{sec:lb}.   For the field $\EF \circ J,$ a direct application of Corollary~\ref{prop:expbound} below and Markov's inequality shows that with probability going to $1,$ $\EF(J(i\zeta_{b^*})) > -C\delta \log N$ for some sufficiently large $C.$  Therefore, upon verifying Assumption~\ref{ass:wusa} for $\ZF,$ the lower bound in Theorem~\ref{thm:main} follows.

\subsection*{Characteristic polynomials}
Let $A_N$ be a unitarily invariant random matrix (chosen with probability proportional to the weight $e^{-N\tr(V(A_N))}$ on the space of $N\times N$ Hermitian matrices).  We consider $\ESD$ to be the empirical spectral measure of the  matrix $A_N$ and take $\LSD$ to be the corresponding equilibrium measure. By \cite[Theorem~2.1]{Kurt98}, $\LSD$ is the (unique) probability measure which minimize the energy functional
$$
\mathcal{I}_V(\mu) = \iint_{\R^2} \left( \log|t-u|^{-1} + \frac{V(t)+V(u)}{2} \right) \mu(dt)\mu(du) . 
$$
Moreover, it is absolutely continuous with respect to the Lebesgue measure: $\LSD(du)= \LSD(u) du$, and 
 if the potential $V$ is real-analytic and regular, by \cite[(1.6)--(1.8)]{DKMVZ}, we have 
\begin{equation}\label{equilibrium}
\LSD(u) = h(u)\1_{ \mathbf{J}}(u) \prod_{i=0}^m \sqrt{ (a_{i+1}-u)(u-b_i)}  
\end{equation}
where the function $h$ is real-analytic, strictly positive on
\begin{equation}\label{J}
 \mathbf{J}=  (b_0,a_1) \cup (b_1,a_2)\cup \dots \cup (b_m,a_{m+1}) ,
\end{equation}

 and  the intervals $  (b_0,a_1), \dots$, $(b_m,a_{m+1})$ are disjoint.

In the previous sections, we have outlined an approach to controlling the maximum of \begin{equation} \label{eq:EF}
  \EF(q) = \log| \det(q- A_N)| - \Exp\big[ \log| \det(q- A_N) |\big]  
\end{equation}
which reduces the problem to estimating mixed moments of $| \det(q- A_N)|^{\pm 2}$ for various points $q$ near the bulk $\mathbf{J}$.  We now elaborate on how to do such estimates when $V$ is a regular, real-analytic potential.

Besides the asymptotics of the monic orthogonal polynomials with respect to measure $e^{-NV(x)}dx$ that are given in \cite{DKMVZ}, we use one other random matrix tool: the Fyodorov-Strahov formula \cite{FS}.  This formula, which we will introduce presently, allows for expectations of ratios of characteristic polynomials to be expressed in terms of determinants involving the orthogonal polynomials and their Cauchy transforms. %

For any $\ell , k\ge 1$ and $\q = (q_1, \dots , q_\ell) \in \C^\ell$, we define the matrix
\begin{equation} \label{V_matrix}
 \V_k(\q) =  \begin{pmatrix}
1  & q_1 & \hdots & q_1^{k-1}  \\
\vdots & & & \vdots \\
1 & q_\ell &\hdots & q_\ell^{k-1}  
 \end{pmatrix} . 
\end{equation}
In particular, we denote Vandermonde matrix, $\V(\q) = \V_\ell(\q)$,  and 
its determinant
\begin{equation} \label{V_det}
\Delta(\q) := \big| \V(\q) \big| =  \prod_{1 \le i <j \le k} (q_j-q_i) . 
\end{equation}
Let $\pi_n$ be the monic orthogonal polynomial of degree $n$ with respect to the weight $e^{-NV(x)}$ on $\R.$  We define the normalizing constants $(\gamma_n)_{n\ge 0}$ of these polynomials by the formula
\begin{equation}
  \label{eq:gamman}
 \int \pi_n(x) \pi_m(x)   e^{-NV(x)} dx = \gamma_n^{-2} \delta_{nm }.
\end{equation}
For any $q \in \C \setminus \R$, we introduce the Cauchy transform 
\begin{equation} \label{h_n}
h_n(q) = \frac{1}{2\pi i} \int \frac{\pi_n(x)}{x-q}   e^{-NV(x)} dx . 
\end{equation}
The normalizing constants $\gamma_n$ also play a role in determining the $3$-term recurrence for $\pi_n.$ Specifically, there exists a sequence  $\beta_0, \beta_1, \dots \in \R$ such that
\begin{equation} \label{OP_recurrence}
 \pi_{n+1}(x) + \left(\frac{\gamma_{n-1}}{\gamma_{n}}\right)^2  \pi_{n-1}(x)  = (x- \beta_{n})  \pi_n(x) ,
\end{equation}
with $\pi_{-1} \equiv 0$.
By taking the Cauchy transform of both sides, it can be seen that $h_n$ necessarily satisfies the same recurrence for all $n \geq 0.$
One should also keep in mind that, even if we do not write it explicitly, all these sequences $\pi_1, \pi_2, \dots$ and $h_0, h_1, \dots$ depend on the dimension $N$ as the weight $e^{-NV(x)}$ is varying.

The Fyodorov-Strahov formula (\cite[(8)]{FS}) states that, for any $\ell, k \ge 0$, 
\begin{equation} \label{Fyodorov_Strahov}
 \E\left[  \frac{  \prod_{i=1}^{\ell} \det(p_i-A_N)}{ \prod_{j=1}^{k}  \det(q_{j}-A_N)} \right] 
 = \frac{\prod_{j=1}^k(- 2\pi i \gamma_{N-j}^2)}{ {(-1)^{\binom{k}{2}}}   \Delta(q)\Delta(p)}
 \begin{vmatrix}  
  {h}_{N-k}(q_1)  & \hdots & h_{N+\ell-1}(q_1)  \\
 \vdots  & & \vdots\\
 {h}_{N-k}(q_k)  & \hdots  & h_{N+\ell-1}(q_k)  \\
 {\pi}_{N-k}(p_1)  & \hdots  & \pi_{N+\ell-1}(p_1) \\
 \vdots   & &  \vdots\\
 {\pi}_{N-k}(p_\ell)  & \hdots  & \pi_{N+\ell-1}(p_\ell)  \\
 \end{vmatrix}.
 \end{equation}
Care should be taken in the sign conventions in comparing this formula to the original.
By applying the three-term recurrence, this determinant can be reduced to an expression involving only $\left\{ \pi_N, \pi_{N-1}, h_N, h_{N-1} \right\},$ at least in the $k=\ell$ case (see Lemma~\ref{lem:FSbalanced} below). Note that, if $k \neq \ell$, this expression would involve a sum of determinants and we will avoid this case by considering only balanced ration (i.e.~$\ell=k$).
This reduction is useful due to the nature of the \cite{DKMVZ} asymptotics.  While the Riemann-Hilbert steepest descent method can handle joint asymptotics of 
\[
  \left\{ (\pi_n, h_n) : N- \kappa \leq n \leq N \right\}
\]
for fixed $\kappa \in\N$ (see \cite[Above Theorem 1.1]{DKMVZ}), it is typically formulated just for the matrix 
\begin{equation} 
  \label{Y_matrix}
  Y_N(q) = \begin{pmatrix}
    \pi_N(q) & h_N(q) \\  \tilde{\pi}_{N}(q) & \tilde{h}_{N}(q) 
  \end{pmatrix},
  \hspace{.6cm}\text{where}\hspace{.6cm}
  \begin{pmatrix}
    \tilde{\pi}_{N}(p) \\
    \tilde{h}_{N}(q) 
  \end{pmatrix}
  =     -2\pi i \gamma_{N-1}^2
  \begin{pmatrix}
 \pi_{N-1}(p) \\
 h_{N-1}(q)
  \end{pmatrix}.
\end{equation}
This is the unique analytic function in $\C\backslash\R$ so that for all $x\in \R$,
\begin{equation} \label{RHP}
Y_{N,+}(x)= Y_{N,-}(x) 
\begin{pmatrix}
 1 & e^{-NV(x)} \\  0 & 1  
\end{pmatrix}
\end{equation} 
and 
$\displaystyle Y_{N}(q)= \begin{pmatrix}
 q^N & 0 \\  0 & q^{-N}\end{pmatrix} \left( \I  + \underset{q\to\infty}{O}(q^{-1}) \right)$. 
The main delicate point in our application of these asymptotics is to ensure the uniformity of the error estimates required by Assumption~\ref{ass:wusa}.

\subsection*{Organization}
In Section~\ref{sec:opbackground}, we give an overview of the asymptotics of orthogonal polynomials in the plane. We closely follow the argument of \cite{DKMVZ} in the regular case, giving enough details so that the presentation is self-contained. 
In particular, we deduce from these asymptotics, estimates on exponential moments of the log-potential of the field $\EF$. 
In Section~\ref{sec:opmeso}, we show how  to compare mixed exponential moments of $\EF$ to the mixed exponential moments of an idealized Gaussian field $\GF$. Hence, verifying the assumption~~\ref{ass:wusa} and, by applying theorem~\ref{thm:subharmoniclb}, completing the lower-bound in  Theorem~\ref{thm:main}. Section~\ref{sec:matching} contains a key combinatorial estimate needed in Section~\ref{sec:opmeso}. In Section~\ref{sec:ub}, we show the complete upper bound in Theorem~\ref{thm:main} using the estimates from Section~\ref{sec:opbackground}.
Finally in Section~\ref{sec:lb}, we prove the general lower bound Theorem~\ref{thm:subharmoniclb}. \\

\subsection*{Notation}
For an $n \times n$ matrix, we will use the norm
$$
\| A\| = \sup\big\{ |A_{ij}| :  i,j \in \{1,2, \ldots, n\}  \big\} 
$$

We use $\ll,\gg,$ and $\asymp$ notation.  For two functions $f$ and $g$ of some set $X \to \R$ we say $f \ll g$ if there is a constant $c>0$ so that $f(x)  \leq c g(x)$ for all $x \in X.$   In some cases, we may allow the constant to have some parameter dependence, in which case we will explicitly state its dependence or write $f \ll_\epsilon g$ which is to say that for each $\epsilon>0$ there is a constant so that the inequality holds.  The notation $f \gg g$ means $g \ll f,$ and the notation $f \asymp g$ means $f \ll g$ and $f \gg g.$

We additionally use $O(\cdot)$ notation.  For a function $f: X \to \R$, $g=O(f)$ if and only if $|g| \ll f.$  If we wish to restrict or clarify the parameter dependence of $O(\cdot),$ we will display the parameters beneath the $O(\cdot).$

\subsection*{Acknowledgements}
We would like to thank Ofer Zeitouni and Kurt Johansson for helpful conversation.  We would especially like to thank Ofer for pointing out the Fyodorov-Strahov formula which started this project.

\section{OP asymptotics background}
\label{sec:opbackground}

The aim of this section is to review the asymptotics of the solution $Y_N$ of the RHP \eqref{RHP}. We will not need the explicit solution given in \cite{DKMVZ}, but the special form of the solution will be important for us to check the assumption~\ref{ass:wusa} in section~\ref{sec:opmeso}. So we will review in details how to apply  the Deift-Zhou steepest descent method to  the RHP \eqref{RHP},   collecting the estimates which are important along the way.
Recall that by \eqref{equilibrium}--\eqref{J}, the equilibrium measure is supported on the closure of 
\[
 \mathbf{J}=  (b_0,a_1) \cup (b_1,a_2)\cup \dots \cup (b_m,a_{m+1}).
\]
For notational convenience, we set $a_0 = a_{m+1}$  and $\mathbf{J}^* = \R \backslash \overline{\mathbf{J}} $.  Viewing these sets on the Riemann sphere, this makes $(a_0, b_0)$ into an interval containing the point at infinity.    Let 
\begin{equation} \label{g_function}
  g(q) = \int_\R \log(q-u) \LSD(du) .
\end{equation}
Note that this function appears explicitly in the normalization of the field $\EF$,~\eqref{eq:EF}.
Moreover, $g$ is analytic  in $\C\backslash(-\infty, a_{0})$ and it follows from the Euler-Lagrange equation defining the equilibrium measure that  there exists a real constant $\ell_V$  so that the function 
\begin{equation} \label{EL}
 H(x) : = - V(x)- \ell_V+ g_+(x) + g_-(x)  
\end{equation}
satisfies  the conditions $H(x) = 0$ for all $x\in \mathbf{J}$ and, if the potential $V$ is regular, 
\begin{equation}\label{H_estimate}
H(x)<0 , 
\quad\quad x\in \mathbf{J}^* ,
\end{equation}
see \cite[(1.10-1.13)]{DKMVZ}. On the other hand, for any $x\in\R$, we have
$$
g_+(x) - g_-(x)  = 2\pi i \int^{a_0}_x \LSD(du) .
$$
When $V$ is analytic, this function can be analytically continued in both the upper and lower half plane in  a neighborhood of $\mathbf{J}$; \cite[(3.43-3.46)]{DKMVZ}. 
To be more specific, we can express $2\pi i \rho(du) = \sqrt{\mathrm{Q}(u)} du$ for all $s\in \mathbf{J}$ where $\mathrm{Q}(s)$ is real-analytic and $\sqrt{\mathrm{Q}}$ is given by the principal branch;  \cite[(3.3-3.5)]{DKMVZ}. Then, these analytic continuations are given by 
\begin{equation} \label{G_function}
G(z) = \pm   \int^{a_0}_z\sqrt{\mathrm{Q}(u)}du , 
\quad\quad z\in \UH_{\pm} .
\end{equation}
In particular, if the potential $V$ is regular, then $Q(x)<0$ for all $x\in \mathbf{J}$ and all the edge points of $ \mathbf{J}$ are simple zeros of $Q$.
Then a simple argument shows that  for all $x\in \mathbf{J}$, there exists a constant $\mathrm{c}>0$ so that 
\begin{equation} \label{G_estimate}
\pm \Re G(x \pm i y ) > \mathrm{c} y , 
\end{equation}
when $y >0 $ and sufficiently small. 
We are now ready to  present the asymptotics of the matrix $Y_N$, \eqref{Y_matrix}.
We begin by introducing
\begin{equation}  \label{eq:pauli}
  M_N(q) = e^{-N\ell_V/2 \sigma_3} Y_N(q) e^{-N(g(q)-\ell_V/2)\sigma_3},
\end{equation}
where $\sigma_3 =     \begin{pmatrix}
      1 &0\\
     0 & -1 \\
    \end{pmatrix}$. This normalization implies that 
$ M_{N}(q)=  \I  + \underset{q\to\infty}{O}(q^{-1}) $.
Let us {\it deform} the RHP by introducing {\it lens-shaped regions} delimited by smooth arcs $\Sigma_{\pm}$, see~\cite[Figure 5 p.60]{Kuijlaars},  and  a new matrix
\begin{equation} \label{eq:M_matrix}
 M_N(q) = S_N(q) U_N(q) ,
\end{equation}
where 
\begin{equation} \label{eq:U_matrix}
U_N(q) = \begin{cases} 
\I&\text{if } q \text{ is outside the lens-shaped regions} \\
\begin{pmatrix}
 1 & 0 \\  e^{\mp N G(q)}& 1  
\end{pmatrix}
 &\text{if } q \text{ lies inside the $\pm$ lens-shaped regions}
\end{cases} .
\end{equation}
Note that we {\it open the lenses} in such a way that the condition \eqref{G_estimate} holds inside  lens-shaped regions, except in an $\epsilon$-neighborhood of the end-points of $ \mathbf{J}$. In particular, we obtain the estimate:
\begin{equation} \label{U_estimate}
\| U_N(q) - \I \| \le e^{- \mathrm{c} N |\Im q|} , 
\end{equation}
 for all $q\in \C \backslash \R$, $\epsilon$-away from the end-points of $ \mathbf{J}$. 
Using the above definitions, we claim that $S_N$ is the solution of the 
following RHP:
\begin{equation}
  \begin{array}{ll}
     S_N \text{ is analytic in }\C \setminus \big\{ \R\cup \Sigma_{\pm} \big\} , \\
   S_{N,+}(q) =S_{N,-}(q)\nu_N(q) , &q\in\mathbf{J}^*
     \\
   S_{N,+}(q) =S_{N, -}(q)
    \begin{pmatrix}
     0  & 1\\
     -1 & 0 \\
    \end{pmatrix}, 
     &q\in\mathbf{J} \\
    S_{N,+}(q) =S_{N,-}(q) 
    \begin{pmatrix}
 1 & 0 \\  e^{\mp N G(q)}& 1  
\end{pmatrix},
 &q\in\Sigma_{\pm}\\
    S_N(q) \to  \I , &\text{as }q \to \infty 
  \end{array}
  \label{eq:S}
\end{equation}
where 
$$
\nu_N(q) = \begin{pmatrix}
       e^{iN\Omega_j}  & e^{-N H(q)}\\
     0 & e^{-iN\Omega_j} \\
    \end{pmatrix}, \quad q \in (a_j,b_j)  , \quad j=0,\dots, m 
$$
and the  parameters  $\left\{ \Omega_j \right\}_0^m$ are defined by
\[
  \Omega_j = 2\pi  \int_{b_j}^\infty \LSD(du) . 
\]
In particular, we have $\Omega_0=0$, so that the  jump matrix $\nu_N$ is exponentially close to the identity on the interval $(a_0, b_0)$ which contains $\infty$.
We refer to \cite{DKMVZ} sections 3.3 and 4.1 for the details of this construction, or to sections 5.1-5.2 in the lecture notes \cite{Kuijlaars} for a comprehensive presentation.
In particular, in  \cite{Kuijlaars},  the constructions of the global and edge {\it parametrices} are explained in detail in the one-cut regular case.\\

 We are left with the task of finding a solution of \eqref{eq:S}. 
The idea is to disregard the terms which converge to 0 in the jump matrices as $N\to\infty$. By \eqref{H_estimate} and \eqref{G_estimate}, this leads us to consider the following RHP:
\begin{equation}
  \begin{array}{ll}
  M^\infty_N \text{ is analytic in } \C\setminus [b_0, a_0] , \\
 M^\infty_{N,+}(q) =M^\infty_{N,-}(q)\nu_N^\infty(q) , &q\in\mathbf{J}^*
     \\
   M^\infty_{N,+}(q) =M^\infty_{N, -}(q)\sigma,  &q\in\mathbf{J} \\
    M^\infty_N(q) \to  \I , &\text{as }q \to \infty  
  \end{array}
  \label{eq:M_infty}
\end{equation}
with
$$
\nu^\infty_N(q) = \begin{pmatrix}
       e^{iN\Omega_j}  & 0\\
     0 & e^{-iN\Omega_j} \\
    \end{pmatrix}, \quad q \in (a_j,b_j)  , \quad j=0,\dots, m  ,
$$
and 
$$
\sigma =      \begin{pmatrix}
     0  & 1\\
     -1 & 0 \\
    \end{pmatrix} .
$$
The matrix  $M^\infty_N$ is usually called the {\it global parametrix} and we expect that
$  S_N(q)  \sim  M^\infty_N(q)$ as $N\to \infty$ (at least $\epsilon$-away  from the endpoints of $\mathbf{J}$ where the terms that we neglected do not tend to 0). 
In \cite{DKMVZ}, the matrix  $M^\infty_N$ is constructed in two steps. 
First, we consider the function
\begin{equation}
  \gamma(q) = 
  \biggl[
    \prod_{i=0}^m \frac{q-b_i}{q-a_{i+1}}
  \biggr]^{1/4} .
  \label{eq:gamma}
\end{equation}
The branches are chosen so that $\gamma$ is analytic on $\C \backslash\mathbf{J}^*$, $\gamma(q) \sim 1$ as $q\to\infty$ in $\UH_+$, and it satisfies a jump condition:
\(
\gamma_+(q) = i \gamma_-(q)
\) on $\mathbf{J}^*$, c.f.\,\cite[(4.31)]{DKMVZ}.
  Beware that, to remain consistent with existing literature,  we use the notation $\gamma$ and $\gamma_n,$ emphasizing that with a subscript, we refer to the normalizing constant \eqref{eq:gamman}. In particular, if we define the matrix
\begin{equation}   \label{H_matrix}
  H=
  \begin{pmatrix}
    \tfrac{\gamma + \gamma^{-1}}{2} 
    &
    \tfrac{\gamma - \gamma^{-1}}{-2i}  \\
    \tfrac{\gamma - \gamma^{-1}}{2i} 
    &
    \tfrac{\gamma + \gamma^{-1}}{2}  \\
  \end{pmatrix} , 
 \end{equation}
then it is the (unique) solution of the RHP:
\begin{equation}   \label{eq:H}
  \begin{array}{ll}
 H  \text{ is analytic in } \C\setminus \overline{\mathbf{J}^*} ,\\
    H_+(q) = H_-(q) \sigma^{-1} , &q \in \mathbf{J}^*  \\
    H(q) \to  \I , &\text{as }q \to \infty , q \in \UH_+ \\
    H(q) \to \sigma, &\text{as }q \to \infty , q \in \UH_-,  
  \end{array} 
\end{equation}
with $\text{L}^2$--boundary values.

Note that, by \cite[Lemma 4.1]{DKMVZ}, all the zeros of the functions 
$  \gamma \pm \gamma^{-1}$  are located at points $z_j \in (a_j, b_j)$ for $j=1,\dots, m$ on the $\mp$-side of $\mathbf{J}^*$ respectively. 
To complement this object, we need a matrix of oscillatory functions
$$  \Theta(q) = 
  \begin{pmatrix}
    \vartheta_1(q)
    &
    \vartheta_2(q) \\
    \vartheta_3(q)
    &
    \vartheta_4(q)
  \end{pmatrix}, \quad  q \in \C\setminus \overline{\mathbf{J}^*} 
 $$
 that attains boundary values almost everywhere on $\mathbf{J}^*$ and 
 that solves the RHP:
\begin{equation}
  \begin{array}{ll}
    \Theta    \text{ is analytic in } \C\setminus \overline{\mathbf{J}^*} ,\\
    \Theta_+(q) = \Theta_-(q)
    \left(
    \begin{smallmatrix} 
      0 & 1 \\
      1 & 0 \\
    \end{smallmatrix}
    \right)
    \nu_N^\infty(q) , &q\in\mathbf{J}^* \\
    \Theta (q) = c_\pm
      + O\left(\frac{1}{q}\right), & \text{as } q \to \infty , q \in \UH_\pm
  \end{array}
  \label{eq:Theta}
\end{equation}
for some appropriate constant matrices $c_\pm$
(c.f.\,$M^\#$ from \cite[(4.60-4.64)]{DKMVZ}, which is just $\Theta$ rescaled by a constant diagonal matrix).   

While the solution of this RHP need not be unique (indeed solutions may differ by meromorphic functions with poles on $\mathbf{J}^*$), there is a unique solution so that 
\begin{equation} \label{M_global}
     M^\infty_N(q)=
     \begin{cases}
       \begin{pmatrix}
    \tfrac{\gamma + \gamma^{-1}}{2} \vartheta_1
    &
    \tfrac{\gamma - \gamma^{-1}}{-2i} \vartheta_2 \\
    \tfrac{\gamma - \gamma^{-1}}{2i} \vartheta_3
    &
    \tfrac{\gamma + \gamma^{-1}}{2} \vartheta_4 \\
  \end{pmatrix}   ,
  & \text{if } q\in \UH_+ \\
       \begin{pmatrix}
    \tfrac{\gamma + \gamma^{-1}}{2} \vartheta_1
    &
    \tfrac{\gamma - \gamma^{-1}}{-2i} \vartheta_2 \\
    \tfrac{\gamma - \gamma^{-1}}{2i} \vartheta_3
    &
    \tfrac{\gamma + \gamma^{-1}}{2} \vartheta_4 \\
  \end{pmatrix}  \sigma^{-1} ,
  & \text{if } q\in \UH_-  
     \end{cases} . 
  \end{equation}
  solves \eqref{eq:M_infty} with $\text{L}^2$--boundary values. 

  The entries $\vartheta_1, \dots , \vartheta_4$ are certain explicit ratios of  $\theta$-functions whose arguments depend on the parameters $N$ and  $\left\{ \Omega_j \right\}_1^m$. Moreover, by \cite[Lemma 4.1]{DKMVZ}, the poles of these $\vartheta_i$, $i \in \{1,2,3,4\},$ are located at the points $\{z_j\}_{j=1}^m$ and exactly cancel the zeroes of the functions $\gamma \pm \gamma^{-1}$. Thus, besides the jump conditions,  the entries of the matrix  $ M^\infty_N$ have at worst $1/4$-root singularities at the end-points of~$\mathbf{J}$; c.f.~formulae (4.73)-(4.73)  in \cite{DKMVZ}.  In particular, there exists a constant $C>0$ (independent  of the dimension $N$) so that
 \begin{equation} \label{M_estimate}
\left\|   \begin{pmatrix}
    \tfrac{\gamma + \gamma^{-1}}{2} \vartheta_1
    &
    \tfrac{\gamma - \gamma^{-1}}{-2i} \vartheta_2 \\
    \tfrac{\gamma - \gamma^{-1}}{2i} \vartheta_3
    &
    \tfrac{\gamma + \gamma^{-1}}{2} \vartheta_4 \\
  \end{pmatrix}  \right\| \le C \big( \mathrm{R}(q)+1\big)
\end{equation}
 uniformly over $\C\backslash \mathbf{J}^*$ with
 \begin{equation} \label{R_function}
 \mathrm{R}(q) :=\prod_{i=0}^m |(q-a_i)(q-b_i)|^{-1/4}.
 \end{equation}

An important consequence of this observation is the next powerful identities.
 \begin{lemma}
  All the matrices $Y_N$, $M_N$ and $M^\infty_N$ have determinant identically $1.$
  \label{lem:magicdeterminants}
\end{lemma}
\begin{proof}
  For $Y_N,$ this can be seen from the defining Riemann-Hilbert problem, which shows that $\det Y_N$ extends to an entire function and tends to $1$ at infinity.  The result then holds for $\det M_N$ as well, by formula \eqref{eq:pauli}.
  Similarly, it follows from \eqref{eq:M_infty} that the function $q\mapsto \det M^\infty_N(q)$ has no jump in $\C$ and the estimate \eqref{M_estimate} implies that  all its plausible singularities at the endpoints of $\mathbf{J}$ are removable.  
  Thus, the function $ \det M^\infty_N(q)$  is also entire and tends to 1 at infinity.
\end{proof}

 Let us now give some precise statements about the asymptotics of  $M_N$, \eqref{eq:pauli}. In fact, since the matrices $U_N$ and $\nu_N$ are not asymptotic to $\I$ and $\nu_N^\infty$  uniformly on $\Sigma_{\pm}$ and $\mathbf{J}^*$ respectively (the problem coming from the edge-points), we need to introduce one more auxiliary RHP to get the full solution. Let  $\epsilon>0$ be a small parameter and 
 $\mathscr{E} = \bigcup_{\gamma \in \{ a_j, b_j\}} \D(\gamma, \epsilon) $.
We also let $\mathscr{C}_\epsilon = \{\mathbf{J}^* \cup   \Sigma_{\pm}\}  \backslash \mathscr{E} $ and define
\begin{equation} \label{eq:R_matrix}
 R_N(q) =
 \begin{cases}
   S_N(q) (P_N)^{-1} , & q\in \mathscr{E}  \\
  S_N(q) (M^\infty_N)^{-1} ,  &q\in \C \setminus \big\{  \R\cup \Sigma_{\pm}\cup\overline{\mathscr{E}} \big\}
  \end{cases} .
\end{equation}
In formula \eqref{eq:R_matrix},   $P_N$ is the so-called {\it edge-parametrix}, it is defined so that it has the same jumps as $S_N(q)$ inside $\mathscr{E}$ and it satisfies 
\begin{equation} \label{eq:BC} 
M^\infty_N(q)P_N(q)^{-1} = \I + \underset{N\to\infty}{O}(N^{-1}) 
\end{equation}
uniformly for all $q\in \partial\mathscr{E}$.
The boundary condition \eqref{eq:BC} does not uniquely determine the matrix $P_N$, but it determines its  leading asymptotic; c.f.~formulae \eqref{eq:P}--\eqref{Psi_asymptotics} below.
  Then, $ R_N$ satisfies the RHP:
\begin{equation}
  \begin{array}{ll}
     R_N \text{ is analytic in }\C \setminus \big\{ \mathscr{C}_\epsilon \cup  \partial\mathscr{E} \big\} , \\
   R_{N,+}(q) =R_{N,-}(q)E_N(q) , &q\in\mathscr{C}_\epsilon   \\
  R_{N,+}(q) =R_{N,-}(q) M^\infty_N(q)  P_N(q)^{-1} ,
 &q\in \partial\mathscr{E} \\
    R_N(q) \to  \I , &\text{as }q \to \infty 
  \end{array}
  \label{eq:R}
\end{equation}
where the jump matrix is given by
$$
E_N(q) =  
\begin{cases}
    M^\infty_N(q)\begin{pmatrix}
 1 & 0 \\  e^{\mp N G(q)}& 1  
\end{pmatrix}       M^\infty_N(q)^{-1}  ,  & q\in\Sigma_{\pm} \\
    M^\infty_{N,-}(q)\begin{pmatrix}
      1 &e^{iN\Omega_j-N H(q)}\\
     0 & 1 \\
    \end{pmatrix}       M^\infty_{N-}(q)^{-1}  ,  & q \in \mathbf{J}^*_\epsilon 
\end{cases} . 
$$
In particular, using the conditions \eqref{H_estimate} and \eqref{G_estimate}, an easy computation shows that there exists $\alpha_\epsilon>0$  so that
$$
E_N(q) = \I + \underset{N\to\infty}{O}\big( \| M_N^\infty \|^2 e^{-\alpha_\epsilon N}\big) , 
$$
uniformly for all $q\in \mathscr{C}_\epsilon$. Thus, all the jumps in the RHP \eqref{eq:R} converge uniformly to the identity and it follows from the general theory that
\begin{equation}\label{R_estimate}
R_N(q) = \I + \underset{N\to\infty}{O}(N^{-1}) , 
\end{equation}
uniformly for all $q$ in compact subsets of $\C\backslash\mathbf{J}^*$; c.f.~\cite[section~4.6]{DKMVZ}. In particular, note that the estimate \eqref{R_estimate} is valid everywhere except on $\overline{\mathbf{J}^*}$ because we are free to  deform slightly the jump contours $\Sigma_{\pm}$ and $\partial\mathscr{E}$. \\

At last, we need to say a few words about the   {\it edge-parametrix}. According to \cite[section~4.3]{DKMVZ}, it has the form for all  $\gamma \in \bigcup_{j=0}^m \{a_j, b_j\}$ and $q\in \D(\gamma, \epsilon)$,
\begin{equation} \label{eq:P} 
P_N(q) =    \Psi\big(\phi_{N,\gamma}(q)\big)M^\infty_N(q) , 
\end{equation}
where
$$
\Psi(\zeta) := \frac{1}{\sqrt{2}} \begin{pmatrix}
1 & - i \\ -i &1 
\end{pmatrix}  \zeta^{\sigma_3/4}  A(\zeta) e^{\frac{2}{3} \zeta^{3/2} \sigma_3} ,
$$
and $A$ is the (unique) solution  of the so-called {\it Airy Riemann-Hilbert problem}.
We refer to the lecture notes \cite[section~2.3]{Kuijlaars} for a nice description and an explicit solution of this problem in terms of the Airy function.
In particular, by inspection of $A$,  \cite[theorem~2.6]{Kuijlaars},  it is easy to verify that there exists a constant $C>0$ such that for all $\zeta \in \C$,
\begin{equation} \label{Psi_estimate}
\| \Psi (\zeta) \| \le C . 
\end{equation}
Moreover, by formula (2.60) therein, we have
\begin{equation} \label{Psi_asymptotics}
\Psi(\zeta) = \mathrm{I} + \underset{\zeta\to\infty}{O}(\zeta^{-3/2})
\end{equation}
in the region $-\pi <\arg \zeta<\pi$. 
The map $\phi_{N,\gamma}$ satisfies the relationship
\[
\phi_{N,\gamma}(q)=  \left( \frac{3N}{4}
  \int_{\gamma}^z\sqrt{\mathrm{Q}(s)}ds  \right)^{2/3} ;
\] 
c.f.~\cite[(4.85)]{DKMVZ} and \eqref{G_function} for the connection to the function $G$. 
Notice that in the regular case,
$\gamma$ is a simple zero of  the real-analytic function $\mathrm{Q}$ and the map 
$\phi_{N,\gamma}$ gives an analytic chart in the disk $\D(\gamma, \epsilon)$. 
In particular, the matrix $A$ is chosen so that $\Psi\big(\phi_{N,\gamma}(q)\big)$
has the same jumps as  $S_N$ inside $\D(\gamma, \epsilon)$; c.f.~ \cite[section~5.4]{Kuijlaars}.  Moreover, the asymptotics \eqref{Psi_asymptotics} guarantee that uniformly for all $|q-\gamma|=\epsilon$, 
\begin{equation}\label{P_asymptotics}
   \Psi\big(\phi_{N,\gamma}(q)\big) =  \mathrm{I} + \underset{N\to\infty}{O}(N^{-1}) . 
\end{equation}
Hence, by~\eqref{eq:P}, the matrix $P_N$ satisfies the boundary condition~\eqref{eq:BC}.

\begin{proposition}  \label{prop:M_estimate}
There exists a constant $C>0$ such that for all $q\in\C\backslash\R$, 
$$ \| M_N(q) \| \le C \big( \mathrm{R}(q)+1\big)  $$
where $\mathrm{R}(q)$ is given by formula \eqref{R_function}. 
\end{proposition}

\begin{proof} Because of the estimate \eqref{M_estimate}, we only need to prove that there exists $C>0$ so that for all $q\in\C\backslash\R$,
$$ \| M_N(q) \| \le C \| M_N^\infty(q) \| .$$
By formulae \eqref{eq:M_matrix} and \eqref{eq:R_matrix}, we have
\begin{equation}   \label{eq:M}
M_N(q) = \begin{cases}
R_N(q) M_N^\infty(q) U_N(q) &\text{if } q\in \C \setminus \big\{  \R\cup \Sigma_{\pm}\cup\overline{\mathscr{E}} \big\} \\
R_N(q) P_N(q) U_N(q) &\text{if } q\in  \mathscr{E}  \\
\end{cases} . 
\end{equation}
So, it suffices to observe that, by \eqref{R_estimate}, the error term $R_N(q)$  is uniformly bounded in both $q$ and $N$ and, by \eqref{G_estimate} and formula \eqref{eq:U_matrix}, 
$\| U_N(q) \| \le 1$ for all $q\in \C\backslash\R$. 
\end{proof}

Finally, we can also state the  asymptotics of the matrix $M_N(q)$ in a neighborhood  in the upper-half plane (resp.~lower-half plane) of a point of the bulk.  

\begin{proposition} 
Let $0<\delta <1/2$ and $x \in \mathbf{J}$. 
Let $ \mathscr{K}_N^{\pm}$ be sequences of compact sets such that
$$
 \mathscr{K}_N^{\pm}  \subset \D(x, N^{-\delta} ) 
 \quad\text{ and }\quad
 \mathscr{K}_N^{\pm} \subset \big\{ \pm(\Im q)  > N^{-1+\delta} \big\} .
$$
Then   
\begin{equation} \label{uniform_continuity}
    M_N(q) = M^\infty_{N,\pm}(x) + \underset{N\to\infty}{O_x}(N^{-\delta}) ,
\end{equation}
  uniformly in  $ \mathscr{K}_N^{\pm}$, where $M^\infty_{N,\pm}$ denotes the boundary values of the  matrix \eqref{M_global} and the error term in formula \eqref{uniform_continuity} is uniform for all $x$ in compact subsets of $\mathbf{J}$.
  \label{prop:multicut}
\end{proposition}

Before proceeding to the proof of proposition~\ref{prop:multicut}, it is useful to mention the uniform continuity property of the matrix $M^\infty_{N}$:

\begin{lemma} \label{lem:uniform_continuity}
Let, for all $q\in \C\backslash \overline{\mathbf{J}^*}$, 
$$   \widetilde{M}^\infty_{N}(q) = 
  \begin{pmatrix}
    \tfrac{\gamma + \gamma^{-1}}{2} \vartheta_1
    &
    \tfrac{\gamma - \gamma^{-1}}{-2i} \vartheta_2 \\
    \tfrac{\gamma - \gamma^{-1}}{2i} \vartheta_3
    &
    \tfrac{\gamma + \gamma^{-1}}{2} \vartheta_4 \\
  \end{pmatrix} . 
$$
For any  $x\in \C\backslash \overline{\mathbf{J}^*}$, we have 
$$
 \widetilde{M}^\infty_{N}(q) = \widetilde{M}^\infty_{N}(x) + \underset{r\to0}{O_x}(r)
$$
uniformly for all $q\in \D(x, r)$. In particular, the error term is independent of $N$ and uniform for all $x$ in compact subsets of $\C\backslash \overline{\mathbf{J}^*}$. 
\end{lemma}

\begin{proof} 
Recall that, by  \eqref{H_matrix} and  \eqref{eq:Theta},  the matrices $H$ and  $\Theta$ are analytic in $\C\setminus \overline{\mathbf{J}^*}$. 
So is the matrix $\widetilde{M}^\infty_{N}(q)$ and Lemma~\ref{lem:uniform_continuity} follows directly from Cauchy's formula and the estimate \eqref{M_estimate}.
\end{proof}

\begin{proof}[Proof of Proposition~\ref{prop:multicut}] When the parameter $N$ is large, by formula \eqref{eq:M}, we have 
$M_N(q) = R_N(q) M_N^\infty(q) U_N(q) $
for all $q\in \mathscr{K}_N^\pm$. 
By \eqref{U_estimate}, if $|\Im q| \ge N^{-1+\delta}  $,  then
$$ U_N(q) = \I + \underset{N\to\infty}{O}\big (e^{- \mathrm{c} N^\delta} \big) .$$ 
Moreover, using the estimates \eqref{R_estimate} and \eqref{M_estimate}, this implies that 
\begin{equation} \label{M_asymptotics}
M_N(q) = M_N^\infty(q) +  \underset{N\to\infty}{O_x}(N^{-1}) ,
\end{equation}
uniformly for all $q \in \mathscr{K}_N^\pm$. 
Notice that, by formula \eqref{M_global}, $M_N^\infty(q) =  \widetilde{M}^\infty_{N}(q)$ for all $q\in \UH_+$. Thus, by assumption, Lemma~\ref{lem:uniform_continuity} implies that
$$
M^\infty_{N}(q) = \widetilde{M}^\infty_{N}(x) + \underset{N\to \infty}{O_x}(N^{-\delta}) ,
$$
uniformly for all $q \in \mathscr{K}_N^+$.
Combined this estimate with formula \eqref{M_asymptotics} and replacing 
 $ \widetilde{M}^\infty_{N}(x) =  M^\infty_{N,+}(x)$, this yields formula~\eqref{uniform_continuity}. 
 To get the estimate for $\mathscr{K}_N^-$, we follow the same argument  except that we must use that $M_N^\infty(q) =  \widetilde{M}^\infty_{N}(q) \sigma$ for all $q\in \UH_-$ and 
 that $ \widetilde{M}^\infty_{N}(x) \sigma =  M^\infty_{N,-}(x)$ for all $x\in \mathbf{J}$. 
\end{proof}

To conclude this section we give an application of Proposition~\ref{prop:M_estimate} to estimate the Laplace transform of the random field $\EF$, \eqref{eq:EF}. 

\begin{corollary}
  There is a constant $C>0$ so that for all $q \in \C\backslash\R$,
  \[
    \Exp e^{\pm2\EF(q)} \leq C\frac{(1+|\Im q|)\mathrm{R}(q)^2}{|\Im q|} ,
  \]
where $\mathrm{R}(q)$ is given by formula \eqref{R_function}.  
 \label{prop:expbound}
\end{corollary}

\proof Using the Fyodorov-Strahov formula, \eqref{Fyodorov_Strahov}, we have
\begin{align*}
    \Exp [e^{2 \EF(q)}  ]
    &=  \E\big[| \det(q- A_N)|^2 \big] e^{-2N \Re g(q)}  \\
    &=  \frac{1}{ \overline{q} -q} 
 \begin{vmatrix}  
 {\pi}_{N}(q)   & \pi_{N+1}(q) \\
 {\pi}_{N}(\overline{q})   & \pi_{N+1}(\overline{q}) 
 \end{vmatrix}  e^{-N \{ g(q) + g(\overline{q})\} } . 
\end{align*}
\noindent Using the recurrence relation \eqref{OP_recurrence}, we can express this determinant in terms of the entries of the matrix $M_N(q)$, \eqref{eq:pauli}, we obtain
\[
    \Exp[ e^{2 \EF(q)}  ]
    = \big| M_N(q)_{11} \big|^2   -  \frac{e^{N\ell}}{2\pi i \gamma_{N}^2 (\overline{q} -q)}
       \begin{vmatrix}  
M_{N}(p)_{21}   & M_{N}(p)_{11} \\
M_{N}(\overline{p})_{21}   & M_{N}(\overline{p})_{11} 
 \end{vmatrix} .
\]

The asymptotics \cite[(1.63)]{DKMVZ} imply that there exists a universal constant $C>0$ such that
\[
 \left| \frac{e^{N\ell}}{2\pi i \gamma_{N}^2}\right| \le C ,
\]
This implies that 
\[
    \Exp[e^{2 \EF(q)}]  \ll \frac{\|M_N(q)\|^2(1+|\Im q|)}{|\Im q|} 
\]
and the claim is a direct consequence of Proposition~\ref{prop:M_estimate}.

The claim for the negative power is similar, as 
\begin{align*}
    \Exp [e^{-2 \EF(q)}  ]
    &=  \E\big[| \det(q- A_N)|^{-2} \big] e^{2N \Re g(q)}  \\
    &=  \frac{1}{ \overline{q} -q} 
 \begin{vmatrix}  
 {h}_{N}(q)   & h_{N+1}(q) \\
 {h}_{N}(\overline{q})   & h_{N+1}(\overline{q}) 
 \end{vmatrix}  e^{N \{ g(q) + g(\overline{q})\} } . 
\end{align*}
\noindent Proceeding in an analogous way as to the positive power, we are led to the conclusion of the corollary. 

\qed\\

\section{Uniform mesoscopic OP asymptotics}
\label{sec:opmeso}

This section is devoted to compute the expectations of balanced ratios of characteristic polynomials, in particular the verification of the assumptions~\ref{ass:wusa} for the field $\EF$, \eqref{eq:EF}. Recall that $J:\D\to\C\backslash[-1,1]$ denotes the Joukowsky transform given by \eqref{eq:Joukowsky}, and we assume that the support of the equilibrium measure  $\mathbf{J} \subseteq [-1/2,1/2]$. 
We will compare locally the field $\ZF = \EF \circ J$ which is almost surely harmonic in the half-disk $\D_\pm$  to the Gaussian field $\GF$, \eqref{covariance_GF}.
Let $x\in\mathbf{J}$ and $\omega\in\partial\D_+$ be the pre-image of $x$  under $J$,  then we define for any $0<\delta<1/2$,
\begin{equation} \label{eq:domain}
  \mathcal{D}_{N,\delta}
  =\left\{ 
    re^{i\theta}\omega
    \ |\ 1 - N^{-\delta} \leq r \leq 1 - N^{-1+\delta},
    |\theta| \leq N^{-\delta}
  \right\}.
\end{equation}

We slightly generalize the setup from Assumption~\ref{ass:wusa} here, as the proof gives slightly more than what is required for Theorem~\ref{thm:subharmoniclb}.
For any $k \in \N$ and $\epsilon > 0$, let $\mathfrak{S}_{k,\epsilon,\delta}$ be the collection of cylinder functions from $\{\D \to \R\} \to \R$ of the form
\[
  \F \mapsto  
  \sum_{z \in \mathbf{z}'} 2\F(z)
  -\sum_{w \in \mathbf{w}'} 2\F(w),
\]
where $\mathbf{z}'$ and $\mathbf{w}'$ are disjoint subsets of $\mathcal{D}_{N,\delta}$ with $|\mathbf{z}'|=|\mathbf{w}'|=k$ and such that 
\begin{equation} \label{condition_1}
\inf\big\{ \dH(u,v) :  u, v\in \mathbf{z}' \cup \mathbf{w}', u\neq v \big\} \ge  \epsilon . 
 \end{equation}
 We now define a family of perturbed biases. Namely, let $\Bias' \in \mathfrak{S}_{\ell,\epsilon,\delta} $  and define $\mathfrak{W}_{\ell,\delta}(\Bias')$ as the set of all cylinder functions
\[
  \F \mapsto \Bias'(\F) + 
  \sum_{z \in \mathbf{z}} 2\F(z)
  -\sum_{w \in \mathbf{w}} 2\F(w),
\]
where $\mathbf{z}$ and $\mathbf{w}$ are disjoint
    subsets of $\mathcal{D}_{N,\delta}$ with  $|\mathbf{z}|=|\mathbf{w}|=\ell$
and the property that:
    if we denote $Z =  \mathbf{z} \cup \mathbf{z}' $ and $W=\mathbf{w} \cup \mathbf{w}' $,
    there is a bijection $\phi : \mathbf{z} \to \mathbf{w}$  so that for all $z\in \mathbf{z}$:
    \begin{equation} \label{condition_2}
      \dH(z,\phi(z))\leq 
       \min\left\{  
      \min_{z' \in Z\setminus \{z\}} 
      \dH(z,z') ;
      \min_{w \in W \setminus \{\phi(z)\}} 
      \dH(w,\phi(z)) \right\}.
    \end{equation}

We are now ready to state the main result of this section:

\begin{proposition}
  For all $k,\ell \geq 0$, all $\epsilon > 0$, and all $\delta > 0$ sufficiently small, we have
  \[
    \Exp\left[ e^{\Bias{\ZF}} \right]
    =
    \Exp\left[ e^{\Bias{\GF}} \right]
    \left(1+ \underset{N\to\infty}{O}(N^{-\delta})\right) ,
  \]
  uniformly in $\Bias' \in \mathfrak{S}_{k,\epsilon,\delta}$ and uniformly in $\Bias \in \mathfrak{W}_{\ell,\delta}(\Bias')$. 
  \label{prop:ZG}
\end{proposition}
As $\ZF$ is harmonic in $\D,$ this proposition and Theorem~\ref{thm:subharmoniclb} (after rotating the field $\ZF$ to move $\omega$ to $i$) implies the lower bound in Theorem~\ref{thm:main}.

The first step in the proof of proposition~\ref{prop:ZG} is a reduction of the Fyodorov-Strahov formula, \eqref{Fyodorov_Strahov},  in the case that $k=\ell$.
Namely, we show that we can express the expected value of balanced ratio of characteristic polynomials only in terms of the matrix $Y_N$, \eqref{Y_matrix}, whose asymptotics we presented in section~\ref{sec:opbackground}.
 For notational convenience, if $\q=(q_1 \dots, q_\ell)\in (\C\backslash\R)^\ell$  and $f : \C\backslash\R \to \C$ is a continuous function, we let
$$
f(\q) =  \diag( f(q_1), \dots,f(q_\ell)). 
$$
Recall also that $\V(\q)$ and $\Delta(\q)$ denote the Vandermonde matrix and determinant of the points $(q_1 \dots, q_\ell)$; c.f.~formulae \eqref{V_matrix} and \eqref{V_det} respectively. In particular, if $\q= (\mathbf{u}, \mathbf{v})$, we will also use the notation:
$$
\V(\q) = \V(\mathbf{u},\mathbf{v})
\quad\text{ and }\quad
\Delta(\q)= \Delta(\mathbf{u},\mathbf{v}) . 
$$

\begin{lemma} 
Let $\ell \geq 1,$ $\p=(p_1 \dots, p_\ell)$ and $\q=(q_1 \dots, q_\ell)$ be two tuples of  distinct points in $\C\backslash\R$. We have for all $N \geq \ell,$
  \[
    \E\left[  \frac{  \prod_{i=1}^{\ell} \det(p_i-A_N)}{ \prod_{j=1}^{\ell}  \det(q_{j}-A_N)} \right] 
    = \frac{1}{ \Delta(\q)\Delta(\p)}
    \begin{vmatrix}
    \tilde{h}_{N}(\q)  \V(\q) &  h_N(\q)  \V(\q) \\
     \tilde{\pi}_{N}(\q) \V(\p) &    \pi_N(\q)  \V(\p) \\
    \end{vmatrix} .
  \]
  \label{lem:FSbalanced}
\end{lemma}
\begin{proof}
For any $n\in \N$, let
$$
v_N = \big(h_n(q_1),\dots , h_n(q_\ell), \pi_n(p_1),\dots, \pi_n(p_\ell)\big)^t . 
$$

  We begin with  the Fyodorov-Strahov formula, , \eqref{Fyodorov_Strahov},   which we 
  write as
  \[
    \E\left[  \frac{  \prod_{i=1}^{\ell} \det(p_i-A_N)}{ \prod_{j=1}^{\ell}  \det(q_{j}-A_N)} \right] 
    = \frac{\prod_{j=1}^\ell(- 2\pi i \gamma_{N-j}^2)}{ {(-1)^{\binom{\ell}{2}}}   \Delta(\q)\Delta(\p)}
\big| v_{N-\ell}, \dots, v_{N+\ell-1} \big|
  \]
  Since $\pi_n$ and $h_n$ satisfy the same $3$-term recurrence relation, \eqref{OP_recurrence}, we get
  \[
    v_{n+1} + \left(\frac{\gamma_{n-1}}{\gamma_{n}}\right)^2  v_{n-1}  = (A- \beta_{n}I)  v_n,
  \]
  where $A = \diag(q_1,\dots, q_\ell,p_1, \dots, p_\ell).$
  Hence, by induction,  we can conclude that for each $j \geq 1$
  \begin{align*}
    &v_{N+j} - A^jv_{N}  \in \Span\{ v_{N-1}, v_N, Av_N, \dots, A^{j-1}v_N\} \\
    &v_{N-j} - \left(\frac{\gamma_{N-1}}{\gamma_{N-j}}\right)^2A^{j-1}v_{N-1}  \in \Span\{ v_{N}, v_{N-1}, Av_{N-1}, \dots, A^{j-2}v_{N-1}\}.
  \end{align*}
  Therefore, applying column operations to the determinant, we obtain
  \[
    |v_{N-\ell}, \dots, v_{N+\ell-1}|
    =
\prod_{j=1}^\ell  \left(\frac{\gamma_{N-1}}{\gamma_{N-j}}\right)^2
    |A^{\ell-1}v_{N-1}, \dots, v_{N-1}, v_N, Av_{N}, \dots, A^{\ell-1}v_{N}|
  \]
  and 
  \[
    \E\left[  \frac{  \prod_{i=1}^{\ell} \det(p_i-A_N)}{ \prod_{j=1}^{\ell}  \det(q_{j}-A_N)} \right] 
    = \frac{(- 2\pi i \gamma_{N-1}^2)^\ell}{ {(-1)^{\binom{\ell}{2}}}   \Delta(\q)\Delta(\p)}
   \big|A^{\ell-1}v_{N-1}, \dots, v_{N-1}, v_N, \dots, A^{\ell-1}v_{N}\big| .
  \]
  Hence, by definition of  $\tilde{h}_{N}$ and $\tilde{\pi}_{N}$,  c.f.~formula \eqref{Y_matrix}, scaling the factors $-2\pi i \gamma_{N-1}^2$ and fully permuting the first $\ell$ columns,  we have arrived at the claimed formula.
\end{proof}

With these preliminaries, we now turn to the intended application of proving Proposition~\ref{prop:ZG}.  Suppose we have a biasing function
\[
  \Bias(\F) = 
  \sum_{z\in Z} 2\F(z)
  -\sum_{w\in W} 2\F(w) ,
\]
for sets $Z, W \subset   \mathcal{D}_{N,\delta}$ such that $|Z|=|W| = \ell/2 \in \N$. 
We define 
\begin{equation} \label{pq}
 \p= J( \overline{Z} \cup Z)
\quad\text{ and }\quad
\q = J( \overline{W} \cup W) 
\end{equation}  
where, as usual, $J$ is the Joukowsky transform. Then, since  $\ZF = \EF \circ J$,  we have the identity
\begin{equation*}
  \Exp\left[ e^{\Bias(\ZF)} \right]
    =
    \E\left[  \frac{  \prod_{i=1}^{\ell} \det(p_i-A_N)e^{-Ng(p_i)}}{ \prod_{j=1}^{\ell}  \det(q_{j}-A_N)e^{-Ng(q_j)}} \right].
\end{equation*}
Applying Lemma~\ref{lem:FSbalanced}, we obtain
\begin{equation*}
  \Exp\left[ e^{\Bias(\ZF)} \right]
  =
  \frac{1}{ \Delta(\q)\Delta(\p)}
  \begin{vmatrix}
   \tilde h_N e^{Ng}(\q) \V(\q) & h_Ne^{Ng}(\q) \V(\q) \\
    \tilde \pi_Ne^{-Ng}(\p) \V(\p) & \pi_Ne^{-Ng}(\p) \V(\p) \\
  \end{vmatrix} .
\end{equation*}

At this point, it is worth recalling the definition of the normalized matrix $M_N$, \eqref{eq:pauli}.   Namely, after rescaling the rows and columns of the previous determinant, we have 
\begin{equation}
  \Exp\left[ e^{\Bias(\ZF)} \right]
  =
  \frac{1}{ \Delta(\q)\Delta(\p)}
  \begin{vmatrix}
    (M_N)_{22}(\q) \V(\q) &(M_N)_{12}(\q) \V(\q) \\
   (M_N)_{21}(\p) \V(\p) & (M_N)_{11}(\p) \V(\p) \\
  \end{vmatrix},
  \label{eq:scaleddet}
\end{equation}
where, according to our notation, $(M_N)_{ij}(\q)$ are diagonal matrices. 
Then, we can expand combinatorially the RHS of formula \eqref{eq:scaleddet} using the following general identity:

\begin{lemma} \label{thm:error}
Let $\mathrm{A}, \mathrm{B},\mathrm{C}$ and $\mathrm{D}$ be $\ell \times \ell$ diagonal matrices. For any $\p ,\q\in\C^\ell$, we have
\begin{equation*}
\begin{vmatrix} 
\mathrm{A} \V(\q)  &  \mathrm{B}  \V(\q)   \\
 \mathrm{C}\V(\p)   &  \mathrm{D} \V(\p) 
\end{vmatrix}
= \sum_{\substack{ S, T \subseteq  [\ell] \\ |T|+|S| = \ell}}
 \Delta(\q_S, \p_T)   \Delta(\q_{S^*},\p_{T^*})  
 \prod_{s\in S} \mathrm{A}_{ss}  \prod_{s\in S^*}  \mathrm{B}_{ss} \prod_{t\in T} \mathrm{C}_{tt}  \prod_{t\in T^*} \mathrm{D}_{tt} . 
 \end{equation*}
where we denote $T^* = [\ell]\backslash T$, $S^* = [\ell]\backslash S$,
 $\p_T=(p_t)_{t\in T}$, and similarly for the tuples  $\p_{T^*}$, $\q_S$ and $\q_{S^*}$. 
\end{lemma}
 
 \begin{proof}
  Let  $\tilde{\mathrm{A}}, \tilde{\mathrm{B}} \in \mathbb{M}_{2\ell\times \ell}$.  
 For any subset $X \subset [2\ell]$, we denote  $\tilde{\mathrm{A}}_X = (\tilde{\mathrm{A}}_{ij})_{i\in X ,  j \in [\ell] }$ and $\tilde{\mathrm{B}}_{X^*} = (\tilde{\mathrm{B}}_{ij})_{ i\notin X , j \in [\ell]}$.
  Using Laplace's formula, we immediately see that
\begin{equation} \label{det_1}
 \det  \begin{pmatrix}
  \tilde{\mathrm{A}} &\tilde{\mathrm{B}}
  \end{pmatrix} =
 \sum_{\begin{subarray}{c} X  \subseteq  [2\ell] \\  |X|= \ell  \end{subarray}} 
 \det (\tilde{\mathrm{A}}_X)  \det (\tilde{\mathrm{B}}_{X^*}) . 
 \end{equation}
Moreover, for any subsets $S, T  \subseteq  [\ell]$, if
 $$
 X = S  \cup \{ \ell+ t : t\in T \}   
 \quad\text{ and }\quad
 \tilde{\mathrm{A}} = \begin{pmatrix} \mathrm{A} \V(\q) \\  \mathrm{C}\V(\p)  \end{pmatrix},
 $$
since $\mathrm{A}$ and $\mathrm{C}$ are diagonal matrices, we check that 
$$
\det (\tilde{\mathrm{A}}_X)  =  \Delta(\q_S,\p_T) \prod_{s\in S} \mathrm{A}_{ss} \prod_{t\in T} \mathrm{C}_{tt} . 
$$
Similarly, we have
 $$
 X^* = S^*  \cup \{ \ell+ t : t\in T^* \}   
 \quad\text{ and }\quad
 \tilde{\mathrm{B}} = \begin{pmatrix} \mathrm{B} \V(\q) \\  \mathrm{D}\V(\p)  \end{pmatrix},
 $$
so that
$$
\det (\tilde{\mathrm{B}}_{X^*})  = \Delta(\q_{S^*},\p_{T^*}) \prod_{s\in S^*} \mathrm{B}_{ss} \prod_{t\in T^*} \mathrm{D}_{tt} . 
$$
To complete the proof, it remains to observe that in  formula \eqref{det_1}, summing over all subset $X  \subseteq  [2\ell]$ with cardinal  $|X|= \ell$, is equivalent to sum over all pair $(S, T)  \subseteq  [\ell]^2$ such that $|S| + |T|= \ell$.  
\end{proof}

The point of this expansion is that all the terms can  be controlled uniformly.

\begin{proposition}
  For all $k,\ell \geq 0$, all $\epsilon > 0$, and all $\delta > 0$ sufficiently small, there is a constant $C>0$ so that uniformly in $\Bias' \in \mathfrak{S}_{k,\epsilon,\delta}$, uniformly in $\Bias \in \mathfrak{W}_{\ell,\delta}(\Bias')$ and uniformly in $S,T \subseteq [2(k+\ell)]$ with $|S|+|T| = 2(k+\ell),$
\begin{equation}\label{bound}
 \left|  \frac{ \Delta(\q_S, \p_T)   \Delta(\q_{S^*},\p_{T^*})  }{\Delta(\q)\Delta(\p)} \right|  \leq C \Exp[ e^{\Bias(\GF)}].
\end{equation}
  \label{prop:bound}
\end{proposition}

Once the asymptotics of the matrix $M_N$ are known, this bound is the main technical task to obtain proposition~\ref{prop:ZG}. Its proof relies on both conditions \eqref{condition_1} and \eqref{condition_2} and is based on a rather sophisticated combinatorial matching between the terms in the nominator and denominator of the LHS of   \eqref{bound}.  For this reason, we postpone it to the section~\ref{sec:matching}. 

The relation between the left hand side and right hand side of \eqref{bound} may not be clear at first sight.  Observe that the next lemma shows that there is essentially equality (with constant $C=1$) in \eqref{bound} if the sets $T$ and $S$ are chosen so that $\q_S= \q_+$ and $\p_T = \p_-$ where
\begin{equation}
\begin{array}{cccc}
\p_+ = J(\overline{Z}), & \p_- =J(Z),& \q_+ = J(\overline{W}), & \q_- =J(W) 
\end{array}.
\end{equation}
In particular, these notation are chosen so that $\p_+, \q_+ \subset \UH_+$  and 
$\p_-, \q_- \subset \UH_-$.   

\begin{lemma} \label{thm:Laplace_G}
With the above notation,
    \begin{equation} \label{det_2}
    \Exp\left[ e^{\Bias(\GF)} \right]
    =\frac{ \big|  \Delta(\q_+, \p_-)\big|^2 }{\big|\Delta(\p)\Delta(\q)\big|}
    + \underset{N\to\infty}{O}\left(N^{-\delta}\Exp[e^{\Bias(\GF)}]\right).
  \end{equation}
\end{lemma}
\begin{proof}
First, if we write $\Delta(\p) = \Delta(\p_+;\p_-)$ and $\Delta(q) =\Delta(\q_+;\q_-)$, then  expand the Vandermonde determinants on the RHS of formula \eqref{det_2} into products, we see that
    \begin{equation} \label{det_3}
\frac{ \big|  \Delta(\q_+, \p_-)\big|^2 }{\big|\Delta(\p)\Delta(\q)\big|}
=     \frac{  \prod_{\p_-\times \q_+} |p- q|^2 }
 {\prod_{\p_+\times \p_-}|p-  p'|\prod_{\q_+\times\q_-}|q- q'|} . 
\end{equation}
Second, by formula \eqref{covariance_GF},
\begin{align}
    \Exp\left[ e^{\Bias(\GF)} \right]
   &\notag = \exp\left(2 \Exp\left[ \sum_{Z\times Z} \GF(z)\GF(z') + \sum_{W\times W} \GF(w)\GF(w')    - 2\sum_{Z\times W} \GF(z)\GF(w) \right]\right) \\
    & \label{eq:Laplace_G}
    = \frac{\prod_{Z\times W} |1-z \bar w|^2}{\prod_{Z\times Z} |1-z\bar z'| \prod_{W\times W} |1- w\bar w'|} .
    \end{align}
  By definition of the Joukowsky transform, \eqref{eq:Joukowsky},  we have   
\begin{equation}\label{J_distance}
  |p- q| = |J(z) - J(w)|  = \frac{|z- w||1-z w|}{2|z  w|}.
\end{equation}
In particular, uniformly for all $z,w \in   \mathcal{D}_{N,\delta}$,  
$$
   |p- \bar q| = |1-z \bar w| +  \underset{N\to\infty}{O}(N^{-\delta}).
 $$   
Thus, if we replace all the terms in  formula \eqref{eq:Laplace_G} using this asymptotic expansion, we obtain
    \begin{equation*}
    \Exp\left[ e^{\Bias(\GF)} \right]
    =
    \frac{  \prod_{\p_+\times \q_+} |p-\bar q|^2 }
    {\prod_{\p_+\times \p_+}|p- \bar p'|\prod_{\q_+\times\q_+}|q- \bar q'|}
    + \underset{N\to\infty}{O}\left(N^{-\delta}\Exp[e^{\Bias(\GF)}]\right).
  \end{equation*}
Since $\overline{\q_+} = \q_-$ and  $\overline{\p_+} = \p_-$, the claim follows immediately from formula \eqref{det_3}.
\end{proof}

We are now in position to complete the proof of proposition~\ref{prop:ZG}. 
In fact, equipped with the previous lemmas, this just boils down to basic linear algebra.

\begin{proof}[Proof of Proposition~\ref{prop:ZG}]
  We start by applying  lemma~\ref{thm:error} to formula  \eqref{eq:scaleddet}, this leads us to 
  \begin{align} \notag
    \Exp\left[ e^{\Bias(\ZF)} \right]
    = &\sum_{\substack{ S, T \subseteq [\ell] , \\ |T|+|S| = \ell}}  \frac{\Delta(\q_S, \p_T)   \Delta(\q_{S^*},\p_{T^*})  }{ \Delta(\q)\Delta(\p)}  \\
    &\label{det_4}  \times    \prod_{q\in S} (M_N)_{22}(q)  
    \prod_{q\in S^*} (M_N)_{12}(q)
    \prod_{p\in T} (M_N)_{21}(p)
    \prod_{p\in T^*} (M_N)_{11}(p)  .  
  \end{align}
  For any $0<\delta <1/2$, the image of the sets  $\mathcal{D}_{N,\delta}$ and $\overline{\mathcal{D}_{N,\delta}}$ under the Joukowsky transform are compact sets $\mathscr{K}_N^\pm$ which satisfy the assumption of proposition~\ref{prop:multicut}. Hence, we can replace the entries of $M_N$ in formula \eqref{det_4} by $M^\infty_{N,\pm}(x)$ up to an error of order $N^{-\delta}$.  Moreover, by proposition~\ref{prop:bound},  this error is uniformly controlled by some constant multiple of $\Exp[e^{\Bias(\GF)}].$ Then, using lemma~\ref{thm:error} 
{\it backward}, we obtain
\begin{equation}\label{det_5}
    \Exp\left[ e^{\Bias(\ZF)} \right]
    =\frac{1}{\Delta(\p)\Delta(\q)}
    \begin{vmatrix}
      \mathrm{M}^+_{22} \V_\ell(\q_+) &
       \mathrm{M}^+_{12} \V_\ell(\q_+) \\ 
       \mathrm{M}^-_{22} \V_\ell(\q_-) &
        \mathrm{M}^-_{12} \V_\ell(\q_-) \\  
      \mathrm{M}^+_{21} \V_\ell(\p_+) &
       \mathrm{M}^+_{11} \V_\ell(\p_+) \\ 
       \mathrm{M}^-_{21} \V_\ell(\p_-) &
        \mathrm{M}^-_{11} \V_\ell(\p_-)
    \end{vmatrix}
    +O\left(N^{-\delta}\Exp[e^{\Bias(\GF)}]\right).
\end{equation}
where we let $\mathrm{M}^\pm_{ij} = \big(M^\infty_{N,\pm}(x)\big)_{ij} $, not to overload the notation. Note that $\mathrm{M}^\pm_{ij}$ are interpreted as diagonal matrices and, for instance, by formula \eqref{V_matrix},
\begin{equation*} 
\mathrm{M}^+_{22} \V_\ell(\q_+) = 
\begin{pmatrix}
M^\infty_{N,+}(x) &&0 \\
& \ddots & \\
0&& M^\infty_{N,+}(x)
\end{pmatrix}
 \begin{pmatrix}
1  & q_1 & \hdots & q_1^{\ell-1}  \\
\vdots & & & \vdots \\
1 & q_{\ell/2} &\hdots & q_{\ell/2}^{\ell-1}  
 \end{pmatrix} . 
\end{equation*}
By definition of $M^\infty_N$, formula~\eqref{M_global}, we have the relations:
$$\begin{array}{llll}
 \mathrm{M}^-_{22} = - \mathrm{M}^+_{21} , &
 \mathrm{M}^-_{12} =- \mathrm{M}^+_{11} , &
  \mathrm{M}^-_{21} =   \mathrm{M}^+_{22}, &
  \mathrm{M}^-_{11} = \mathrm{M}^+_{12} 
 \end{array}.
$$
So, after rearranging (by permuting certain rows) formula \eqref{det_5}, we obtain 
  \begin{equation}\label{det_6}
    \Exp\left[ e^{\Bias(\ZF)} \right]
    =\frac{1}{\Delta(\p)\Delta(\q)}
    \begin{vmatrix}
      \mathrm{M}^+_{22} \V_\ell(\q_+) &
       \mathrm{M}^+_{12} \V_\ell(\q_+) \\
       \mathrm{M}^+_{22} \V_\ell(\p_-) &
        \mathrm{M}^+_{12} \V_\ell(\p_-)\\        
      \mathrm{M}^+_{21} \V_\ell(\p_+) &
       \mathrm{M}^+_{11} \V_\ell(\p_+) \\
           \mathrm{M}^+_{21} \V_\ell(\q_-) &
       \mathrm{M}^+_{11} \V_\ell(\q_-)  
    \end{vmatrix}
    +O\left(N^{-\delta}\Exp[e^{\Bias(\GF)}]\right).
 \end{equation}
Then, we use the factorization
$$
    \begin{pmatrix}
      \mathrm{M}^+_{22} \V_\ell(\q_+) &
       \mathrm{M}^+_{12} \V_\ell(\q_+) \\
       \mathrm{M}^+_{22} \V_\ell(\p_-) &
        \mathrm{M}^+_{12} \V_\ell(\p_-)\\        
      \mathrm{M}^+_{21} \V_\ell(\p_+) &
       \mathrm{M}^+_{11} \V_\ell(\p_+) \\
           \mathrm{M}^+_{21} \V_\ell(\q_-) &
       \mathrm{M}^+_{11} \V_\ell(\q_-)  
    \end{pmatrix}
=
    \begin{pmatrix}
      \V_\ell(\q_+, \p_-) & 0 \\
       0 & \V_\ell(\p_+,\q_- ) 
	\end{pmatrix}
    \begin{pmatrix}
      \mathrm{M}^+_{22} \I_\ell &       \mathrm{M}^+_{12} \I_\ell \\       
      \mathrm{M}^+_{21} \I_\ell  &     \mathrm{M}^+_{11} \I_\ell   
    \end{pmatrix}.
$$
Note that the last matrix is the Kronecker product  
$
   \begin{pmatrix}
      \mathrm{M}^+_{22}  &       \mathrm{M}^+_{12} \\       
      \mathrm{M}^+_{21}   &     \mathrm{M}^+_{11} 
    \end{pmatrix} \otimes \I_\ell
$,
so that its determinant is
$ \mathrm{M}^+_{11}  \mathrm{M}^+_{22} -  \mathrm{M}^+_{21} \mathrm{M}^+_{12}$
which is identically equal to 1 by lemma~\ref{lem:magicdeterminants}.
Moreover, 
since  $\V_\ell(\q_+,\p_-)$ and  $\V_\ell(\p_+, \q_-)$ are exactly Vandermonde matrices, we obtain
\begin{align*}
    \begin{vmatrix}
      \mathrm{M}^+_{22} \V_\ell(\q_+) &
       \mathrm{M}^+_{12} \V_\ell(\q_+) \\
       \mathrm{M}^+_{22} \V_\ell(\p_-) &
        \mathrm{M}^+_{12} \V_\ell(\p_-)\\        
      \mathrm{M}^+_{21} \V_\ell(\p_+) &
       \mathrm{M}^+_{11} \V_\ell(\p_+) \\
           \mathrm{M}^+_{21} \V_\ell(\q_-) &
       \mathrm{M}^+_{11} \V_\ell(\q_-)  
    \end{vmatrix}
 =\Delta(\q_+, \p_-)  \Delta(\p_+, \q_-)   
= |\Delta(\q_+, \p_-) |^2
\end{align*}
since  $ \p_- =\overline{\p_+} $, $\q_-= \overline{\q_+}$, and $\ell$ is even. 
Hence, combining formula \eqref{det_6} and lemma~\ref{thm:Laplace_G}, we conclude that
$$
    \Exp\left[ e^{\Bias(\ZF)} \right]
    = \Exp\left[ e^{\Bias(\GF)} \right]
    + \underset{N\to\infty}{O}\left(N^{-\delta}\Exp[e^{\Bias(\GF)}]\right) ,
$$
which is the required asymptotic behavior. 
\end{proof}

\section{Matching lemma} \label{sec:matching}

The purpose of this section is to prove proposition~\ref{prop:bound}. Recall that $Z$ and $W$ are two disjoint sets in the domain $\mathcal{D}_{N,\delta}$, \eqref{eq:domain}, such that $|Z|=|W| = k+\ell$ and which satisfy the conditions 
\eqref{condition_1} and \eqref{condition_2}.  
For notational convenience, instead of viewing $S, T$ as subsets of the integers $[2(\ell+k)]$, we let 
$S \subseteq \overline{W} \cap W $ and 
$T \subseteq  \overline{Z} \cap Z$. We also define
\begin{equation} \label{+-}
\begin{array}{ll}
T_+   =  \{ z \in Z : z \in T \}  \hspace{1cm}  &T_- = \{  z \in Z  :  \bar{z} \in T  \} \\
T_+^* =  \{ z \in Z : z \in T^* \}   &T_-^* = \{ z \in Z : \bar{z} \in T^* \}
\end{array}
\end{equation}
and similarly for $S_+, S_- , S_+^*, S_-^*$.
Note that all these sets lie in the upper-half disk $\D_+$ and that we have the decompositions:
\begin{equation} \label{decomposition_TS}
Z = T_+  \cup  T_+^* = T_-  \cup  T_-^* 
\quad\text{ and }\quad
 W=  S_+  \cup  S_+^* = S_-  \cup  S_-^*  .
\end{equation}

For any function $f :\D^2 \to \R$ and  any finite sets 
$A$ and $B$ of points in $\D$, we denote
$$
f(A,B) = \prod_{z \in A , w\in B} f(z, w)
$$
and 
\begin{equation} \label{L}
L_f(A, B) = \frac{f(A, B) f(A^*, B^*)}{f(A, A^*) f(B,B^*)}
\end{equation}
where $A^*= Z \backslash A$ and $B^* = W\backslash B$ respectively. In the following, we let 
$$
 \Gamma(z, w) = |1- z\bar w | ,
$$
$d_{\mathrm{E}}$ be the Euclidean metric on $\C$, 
and we introduce the \emph{pseudohyperbolic metric} on $\D$ given by
\begin{equation} \label{metric}
\dist(z,w) = \tanh(\dH(z,w)) = \frac{ | z-w |}{|1-z\bar w|}  . 
\end{equation}
It is easy to check that $\dist$ is indeed a metric  which is uniformly bounded by 1.
As before, we let $p = J(z)$ and $q = J(w)$ for $z\in Z$ and $w\in W$ respectively. Finally, recall that, by formula \eqref{J_distance}, if  the parameter $N$ is large, we have
\begin{equation} \label{J_distance_1}
 \frac{1}{4} \le   \frac{ |z- w| }{|p- q|} \le 1
\hspace{.5cm}\text{and}\hspace{.5cm}
 \frac{1}{4} \le   \frac{ |1- z \bar{w}|}{|p- \bar{q}|} \le 1  . 
\end{equation}

Note that expanding the Vandermonde determinant on the LHS of \eqref{bound}, we obtain
\begin{align*}
 \left|  \frac{ \Delta(\q_S, \p_T)   \Delta(\q_{S^*},\p_{T^*})  }{\Delta(\q)\Delta(\p)} \right|
&=    \frac    {\prod_{T \times S}|p-q| \prod_{T^*\times S^*}|p-q|}
    {\prod_{T\times T^*}|p-p'| \prod_{S \times S^*}|q-q'|}   \\
&= L_{d_{\mathrm{E}}}(T, S)  ,
  \end{align*}
using the notation \eqref{L}. Moreover, we can rewrite formula \eqref{eq:Laplace_G}
as
$$
\Exp\big[ e^{\Bias(\GF)} \big] 
 = L_\Gamma(Z, W)
$$
where  by convention:  $Z^*= Z$ and $W^* = W$. 
Thus, we want to demonstrate that there exists a constant $C>0$ so that  uniformly over many choices of bias $\Bias$ and subsets 
$S \subseteq \overline{W} \cap W $, $T \subseteq  \overline{Z} \cap Z$,  we have
\begin{equation} \label{bound_0}
 L_{d_{\mathrm{E}}}(T, S) \le C L_\Gamma(Z, W) .
  \end{equation}

First, using the estimate \eqref{J_distance_1}, the definitions of the metric $\dist$ and \eqref{+-}, it is easy to verify that there exists a constant $C>0$ which only depends on $|Z|$ such that
\begin{equation}\label{bound_1}
L_{d_{\mathrm{E}}}(T, S)  \le C L_\Gamma(T, S)  L_\dist(T_+, S_+) L_\dist(T_-, S_-) .  
\end{equation}

The second step is to compare the quantities $ L_\Gamma(T, S) $ and $L_\Gamma(Z, W) $. Using the relations \eqref{+-}, we have
$$
    L_\Gamma(T,S) =  \frac{\Gamma(T_+, S_-) \Gamma(T_-, S_+) \Gamma(T_+^*, S_-^*) \Gamma(T_-^*, S_+^*)}
    {\Gamma(T_+, T_-^*) \Gamma(T_-, T_+^*) \Gamma(S_+, S_-^*) \Gamma(S_-, S_+^*)} 
    L_\Gamma(T_+, S_+) L_\Gamma(T_-, S_-) .
    $$
On the other hand, using \eqref{decomposition_TS}, we also see that
$$
L_\Gamma(Z, W) =  \frac{\Gamma(T_+, S_-) \Gamma(T_-, S_+) \Gamma(T_+^*, S_-^*) \Gamma(T_-^*, S_+^*)}
    {\Gamma(T_+, T_-^*) \Gamma(T_-, T_+^*) \Gamma(S_+, S_-^*) \Gamma(S_-, S_+^*)}  
   \frac{ L_\Gamma(T_+, S_-^*) L_\Gamma(T_-, S_+^*) }{L_\Gamma(T_+, T_-) L_\Gamma(S_+, S_-)} . 
$$
Hence, several terms cancel and we are left with
\begin{equation} \label{bound_2}
\frac{L_\Gamma(Z, W) }{  L_\Gamma(T,S)}
=   \frac{ L_\Gamma(T_+, S_-^*) L_\Gamma(T_-, S_+^*) }{   L_\Gamma(T_+, S_+) L_\Gamma(T_-, S_-) }
\frac{1}{L_\Gamma(T_+, T_-) L_\Gamma(S_+, S_-)} . 
\end{equation}

We will ultimately show that this quantity is bounded by $1.$  To do this, we begin by using the following combinatorial identities to take advantage of cancellation in~\eqref{bound_2}.

\begin{lemma}
Let $X$ and $Y$ be any finite sets of (distinct) points in $\D$. For any subsets
 $A, B \subseteq X$, $E,F \subseteq Y$, we have 
\begin{align}
\frac{\Gamma(A,E)\Gamma(B^*, F^*)}{\Gamma(A, F^*)\Gamma(B^*,E)} 
&\label{algebra_1}
=\frac{\Gamma(A \cap B , E \cap F)}{\Gamma(A \cap B, E^*\cap F^*)} \frac{\Gamma(A^* \cap B^*, E^*\cap F^*)}{\Gamma(A^* \cap B^*, E \cap F)} , \\
\frac{L_\Gamma(A,E) L_\Gamma(B,F)}{L_\Gamma(A,F^*)L_\Gamma(B,E^*)}
&=\label{algebra_2} \left(\frac{\Gamma(A \cap B , E\cap F) \Gamma(A^* \cap B^* , E^*\cap F^*)}{\Gamma(A \cap B , E^*\cap F^*) \Gamma(A^* \cap B^* , E\cap F)} \right)^2 ,
\end{align}
where $A^* = X\backslash A$, $E^*=Y\backslash E$   and similarly for $B^*$ and $F^*$. 
\end{lemma}

\begin{proof} Both formulae \eqref{algebra_1} and  \eqref{algebra_2} follow from the simple observations that 
$$
\frac{\Gamma(A,E)}{\Gamma(B^*,E)} = \frac{\Gamma(A \cap B , E)}{\Gamma(A^* \cap B^*, E)} 
$$
and 
$$
\frac{L_\Gamma(A,E) L_\Gamma(B,F)}{L_\Gamma(A,F^*)L_\Gamma(B,E^*)}
= \frac{\Gamma(A,E)\Gamma(B^*, F^*)}{\Gamma(A, F^*)\Gamma(B^*,E)}  
 \frac{\Gamma(A^*,E^*)\Gamma(B, F)}{\Gamma(A^*, F)\Gamma(B,E^*)}  .
$$
\end{proof}

Applying formula \eqref{algebra_1} with $A=F = T_+$ and $B= E= T_-$, we obtain
$$
L_\Gamma(T_+, T_-) =  \frac{\Gamma(T_+ \cap T_- , T_+ \cap T_- ) \Gamma(T^*_+ \cap T^*_- , T^*_+ \cap T^*_- ) }{ \Gamma(T_+ \cap T_- , T_+^* \cap T_-^* )^2 } .
$$
Similarly, applying formula \eqref{algebra_2} with $A =T_+$, $B=T_-$, $E=S_+$ and $F=S_-$, we get
$$
\frac{   L_\Gamma(T_+, S_+) L_\Gamma(T_-, S_-) }{ L_\Gamma(T_+, S_-^*) L_\Gamma(T_-, S_+^*) } =  \left(\frac{\Gamma(T_+ \cap T_- , S_+\cap S_-) \Gamma(T_+^* \cap T_-^* , S_+^*\cap S_-^*)}{\Gamma(T_+ \cap T_- , S_+^*\cap S_-^*) \Gamma(T_+^* \cap T_-^* , S_+\cap S_-)} \right)^2 . 
$$

So, if we now denote $A = T_+ \cap T_-$, $A^* = T^*_+ \cap T^*_-$,  $B = S_+ \cap S_-$ and $B^* = S^*_+ \cap S^*_-$, by formula \eqref{bound_2}, this implies that
\begin{align*}
\frac{L_\Gamma(Z, W) }{  L_\Gamma(T,S)}
&= \frac{ \Gamma(A,B^*)^2 \Gamma(A^*,B)^2 \Gamma(A,A^*)^2 \Gamma(B,B^*)^2}{\Gamma(A,B)^2 \Gamma(A^*,B^*)^2 \Gamma(A,A) \Gamma(A^*,A^*) \Gamma(B,B) \Gamma(B^*,B^*)}\\
&= \Exp\left[ \exp \bigg( \sum_{z\in A \cap B} \GF(z)   -\sum_{w\in A^* \cap B^*} \GF(w)  \bigg) \right]  .
\end{align*}
The last equality follows directly from the definition of $\Gamma$ and the covariance structure of the Gaussian field $\GF$, \eqref{covariance_GF}. 
By Jensen's inequality, this implies that $  L_\Gamma(T,S) \le L_\Gamma(Z, W) $ and by formula \eqref{bound_1}, 
\begin{equation}\label{bound_3}
\frac{L_{d_{\mathrm{E}}}(T, S)}{ L_\Gamma(Z, W)  }  \le C  L_\dist(T_+, S_+) L_\dist(T_-, S_-) .  
\end{equation}

Hence, in order to prove \eqref{bound_0}  to complete the proof of proposition~\ref{prop:bound}, 
it remains to show that there  exists a constant  $C>0$ which depends only on the parameters $\epsilon , k ,\ell$ such that
$$
\max \big\{  L_\dist(T_+, S_+) ; L_\dist(T_-, S_-) \big\} \le C .
$$
This is the point of lemma~\ref{thm:matching} below.
In particular, note that the hypotheses on $\Bias' \in \mathfrak{S}_{k,\epsilon,\delta}$ and  $\Bias \in \mathfrak{W}_{\ell,\delta}(\Bias')$,  \eqref{condition_1}--\eqref{condition_2}, imply that after correctly re-ordering the sets $Z$ and $W$,   $Z\times W$ is a pair-configuration for the  \emph{pseudohyperbolic metric} \eqref{metric} in the sense of the following definition.

\begin{definition} \label{def:pair}
Let $\dist$ be a metric such that $\dist(z,w) \le 1$ for all $z,w \in\D$ and let $ \epsilon>0$ be a small constant. 
A {\bf pair-configuration}  is a collection
$$
\mathfrak{U} = \big( (z_1, w_1), \dots, (z_{\ell+k}, w_{\ell+k})  \big) \in (\D\times \D)^{\ell+k}
$$
 such that the following conditions holds:
\begin{align}
&\label{pair_1}
\dist(z_j, w_j) \leq \left(  \hspace{-.1cm} \min_{w\in W\backslash\{w_j\}} \dist(w,w_j)\wedge  \hspace{-.1cm} \min_{z\in Z\backslash\{z_j\}}  \dist(z,z_j)\right) &\text{if }1 \le j\le \ell , \\
&\label{pair_2}
\begin{cases}
 \dist(z_i, z_j) \wedge \dist(w_i, w_j) \ge \epsilon&\text{if }\ell <  i < j  \le \ell+ k  \\
 \dist(z_i, w_j) \ge \epsilon &\text{if }\ell <  i ,  j  \le \ell+ k  
\end{cases} ,
\end{align}
where, as usual $Z=(z_1,\dots, z_{k+\ell})$ and $W=(w_1,\dots, w_{k+\ell})$.
\end{definition}

\begin{lemma} \label{thm:matching}
There exists a constant $C>0$ which only depends on $\epsilon, \ell , k$ such that for any pair-configuration  $\mathfrak{U}=(Z,W)$ and for any subsets 
$T \subseteq Z$, $S\subseteq W$,  we have:
\begin{equation} \label{pair_0}
 L_\dist(T, S) =  \frac{\dist(T, S)\dist(T^*, S^*)}{\dist(T, T^*)\dist(S, S^*)} \le C , 
 \end{equation}
 where  $T^*= Z\backslash T$ and $S^*= W\backslash S$.
\end{lemma}

\proof
In this proof, we will use the notation $d_1 \ll  d_2$ for the existence of a constant $c>0$ which may depend on $\epsilon>0,\ell>0,$ and $k > 0$ so that
$$ 0< d_1 < c\cdot d_2 , $$
and  $d_1 \asymp d_2$ if both $d_1 \ll d_2$ and $d_1 \gg d_2$.
Let us reformulate the problem by introducing the complete graph $\mathcal{G}$ with vertex set $\mathfrak{U}$ including a loop attached to each vertex. We will denote $\u_j = (z_j, w_j)$ the vertices and $\u_i \u_j$ the edges of $\mathcal{G}$.  Given the sets $T \subseteq Z$, $S\subseteq W$, we define for all $i\neq j$, 
\begin{align*}
 \rho(\u_i\u_j) &= \dist(z_i, z_j)^{\1_{(z_i, z_j) \in T \times T^* \cup T^*\times T }} 
 \dist(w_i, w_j)^{\1_{(w_i, w_j) \in S \times S^* \cup S^*\times S}} , \\
  \widehat{\rho}(\u_i\u_j)& = \dist(z_i, w_j)^{\1_{(z_i, w_j) \in T \times S \cup T^* \times S^* }} 
 \dist(z_j, w_i)^{\1_{(z_j, w_i) \in T \times S \cup T^* \times S^*}} , 
\end{align*}
as well as $ \rho(\u_i\u_i)=1$ and $  \widehat{\rho}(\u_i\u_i) = \dist(z_i, w_i)^{\1_{(z_i, w_i) \in T \times S \cup T^* \times S^*}}$. 
In the following, we interpret $\rho,  \widehat{\rho} : E \to \R_+$ as cost functions defined on the edges of the graph $\mathcal{G}$ such that
\begin{equation*}
 L_\dist(T, S) = \prod_{ i, j  \in [k+\ell]}\frac{ \widehat{\rho}(\u_i\u_j) }{\rho(\u_i\u_j) } . 
 \end{equation*}

Hence, the game is to show that
$$ 
\prod_{ i, j  \in [k+\ell]}\widehat{\rho}(\u_i\u_j)  \ll \prod_{ i, j  \in [k+\ell]} \rho(\u_i\u_j) . $$
We will proceed by estimating the contribution from the different types of edges step by step.
Let us first consider the set of vertices
$$\mathfrak{W}=\big\{ \u_j \in T \times S^* \cup T^* \times S : 1 \le j \le \ell \big\}.    $$
We claim that if  $\u_1 \in \mathfrak{W}$, then 
\begin{equation}\label{estimate_1}
\prod_{j=1}^{k+l} \frac{ \widehat{\rho}(\u_1\u_j) }{\rho(\u_1\u_j) } \ll 1 . 
 \end{equation}

Without loss of generality, we assume that  $\u_1\in T\times S^*$. Then, for any $j\in [k+\ell]$, we have
$$
\rho(\u_1\u_j)  = 
\begin{cases}
\dist(w_1, w_j) &\text{if }  \u_j \in T \times S \\
\dist(z_1, z_j) &\text{if }  \u_j \in T^* \times S^* \\  
\dist(z_1, z_j) \dist(w_1, w_j)   &\text{if }  \u_j \in T^* \times S \\
1 &\text{else}
\end{cases}
$$
and 
$$
 \widehat{\rho}(\u_1\u_j) =
\begin{cases}
\dist(z_1, w_j) &\text{if }  \u_j \in T \times S \\
\dist(z_j, w_1) &\text{if }  \u_j \in T^* \times S^* \\  
\dist(z_1, w_j) \dist(z_j, w_1)   &\text{if }  \u_j \in T^* \times S \\
1 &\text{else}
\end{cases} . 
$$

Using the condition  \eqref{pair_1} and the triangle inequality:
\begin{align*}
&\dist(z_1, w_j) \le \dist(z_1, w_1)  + \dist(w_1, w_j) \le 2\dist(w_1, w_j) \\
&\dist(z_j, w_1) \le \dist(z_j, z_1)+ \dist(z_1, w_1) \le 2 \dist(z_1, z_j) .
\end{align*}
This establishes that $\widehat{\rho}(\u_1\u_j) \ll \rho(\u_1\u_j) $ and we obtain formula \eqref{estimate_1}. 

So, if $\mathcal{G}' =( \mathfrak{U}', E')$ is the complete graph (including all the loops) with vertex-set $\mathfrak{U}' :=\mathfrak{U} \backslash \mathfrak{W}$, formula \eqref{estimate_1} implies that 
\begin{equation}\label{estimate_2}
\prod_{ i, j  \in [k+\ell]}\frac{ \widehat{\rho}(\u_i\u_j) }{\rho(\u_i\u_j) } 
\ll   \prod_{\u_i\u_j \in E'} \frac{ \widehat{\rho}(\u_i\u_j) }{\rho(\u_i\u_j) }   . 
\end{equation}

The next reduction step is more sophisticated. Let us split the vertex-set $\mathfrak{U}'$ in  three categories:
$$
 \mathfrak{A} = \big\{ \u_j \in T \times S : 1 \le j \le \ell \big\} \ , \hspace{.8cm}
 \mathfrak{A}^* = \big\{ \u_j \in T^* \times S^* : 1 \le j \le \ell \big\} \ ,
 $$
and $ \mathfrak{V}=\big\{ \u_j :   \ell < j \le \ell + k\} $ so that the edges of $\mathcal{G}'$ are decomposed as
\begin{equation} \label{pair_4}
E'=  \big((\mathfrak{A} \cup \mathfrak{A}^*)\times (\mathfrak{A} \cup \mathfrak{A}^*)  \big) \cup \big((\mathfrak{A} \cup \mathfrak{A}^*)\times  \mathfrak{V}  \big) \cup 
\big( \mathfrak{V} \times \mathfrak{V} \big) . 
\end{equation}

We will proceed by induction to estimate the contribution coming from two vertices $\u_1\in \mathfrak{A}$ and $\u_2\in \mathfrak{A}^*$. Without loss of generality, 
 $|\mathfrak{A}| \ge |\mathfrak{A}^*| $ and provided that $\mathfrak{A}^*$ is not empty, we choose $\u_1$ and $\u_2$ so that
\begin{equation}\label{pair_3}
\dist(w_1,w_2) = \min \big\{ \dist(w_i,w_j) : \u_i \in \mathfrak{A},  \u_j\in \mathfrak{A}^*    \big\} . 
\end{equation}
We will first compare the quantities
 $$
 \prod_{\begin{subarray}{c} \u_j\in \mathfrak{U}'\backslash \mathfrak{V} \\
j >2 \end{subarray}}   \rho(\u_1\u_j)\rho(\u_2\u_j)
=   \prod_{\begin{subarray}{c} \u_j\in \mathfrak{A} \\ j\neq 1 \end{subarray}}   \dist(z_2, z_j) \dist(w_2, w_j)  \prod_{\begin{subarray}{c} \u_j\in \mathfrak{A}^*  \\ j\neq 2 \end{subarray}}
\dist(z_1, z_j) \dist(w_1, w_j)  
$$
and 
\begin{align*}
 \prod_{\begin{subarray}{c} \u_j\in \mathfrak{U}'\backslash \mathfrak{V} \\
j >2 \end{subarray}}  \widehat{\rho}(\u_1\u_j)\widehat{\rho}(\u_2\u_j)  
&  = \prod_{\begin{subarray}{c} \u_j\in \mathfrak{A} \\ j\neq 1 \end{subarray}} \dist(z_1, w_j) \dist(z_j,w_1)  
 \prod_{\begin{subarray}{c} \u_j\in \mathfrak{A}^*  \\ j\neq 2 \end{subarray}} \dist(z_2, w_j) \dist(z_j, w_2)  .
 \end{align*}

A straightforward consequence of conditions \eqref{pair_1} and \eqref{pair_3}  is that
\begin{align*}
 & \dist(w_2, z_j) \asymp  \dist(w_1, w_j)
  \hspace{.4cm}\text{and}  \hspace{.4cm}
   \dist(z_2, w_j) \asymp  \dist(z_1, z_j) 
     \hspace{.5cm}  \forall\ \u_j \in \mathfrak{A}^*   , \\
     & \dist(w_1, z_j) \asymp  \dist(w_2, w_j) 
\hspace{.5cm}\text{and}\hspace{.4cm}  
\dist(z_1, w_j) \asymp \dist(z_2, z_j)
  \hspace{.4cm}  \forall\ \u_j \in \mathfrak{A}  .
\end{align*}
Hence, we obtain
\begin{equation} \label{estimate_6}
 \prod_{\begin{subarray}{c} \u_j\in \mathfrak{U}'\backslash \mathfrak{V} \\
j >2 \end{subarray}}  \widehat{\rho}(\u_1\u_j)\widehat{\rho}(\u_2\u_j)  
\asymp
 \prod_{\begin{subarray}{c} \u_j\in \mathfrak{U}'\backslash \mathfrak{V} \\
j >2 \end{subarray}}  \rho(\u_1\u_j)\rho(\u_2\u_j)  . 
\end{equation}

Now, define the points $z_\#$ and $w_\#$ by 
\begin{equation*}
\dist(z_1,z_\#)= \min\big\{ \dist(z_1,z_j) : \u_j \in \mathfrak{V} \big\} 
\hspace{.5cm}\text{and}\hspace{.5cm} 
\dist(w_1,w_\#)= \min\big\{ \dist(w_1,w_j) : \u_j \in \mathfrak{V} \big\} .
\end{equation*}
Because of the separation conditions \eqref{pair_2},  only one of these two distances
can be smaller than $\epsilon/3$. 
Without loss of generality, let  us assume that it is attained at $z_\#$ and that $z_\#\in T$, in which case we have
\begin{equation}\label{estimate_5}
\prod_{u_j \in \mathfrak{V}}\rho(\u_1\u_j)\rho(\u_2\u_j)  
\asymp \dist(z_2,z_\#). 
\end{equation}

If $ \dist(w_1, z_\#)\le 2 \dist(z_2,z_\#) $, by \eqref{pair_1}, we have
\begin{equation}\label{estimate_7}
\dist(w_1, z_\#) \dist(z_1, w_1) \dist(z_2, w_2) \le 2\dist(z_2, z_\#) \dist(z_1, z_2) \dist(w_1, w_2) 
\end{equation}
On the other hand,  if $\dist(z_2,z_\#)< \dist(w_1, z_\#)/2$, by the triangle inequality
$$
 \dist(w_1, z_\#) \le \dist(z_\#, z_2) + \dist(z_2, z_1) +  \dist( z_1, w_1)   
 $$
and 
$$
 \dist(w_1, z_\#) \le 4 \dist(z_2, z_1) . 
$$
This implies that the estimate  \eqref{estimate_7} still holds with an extra factor of $2$. The bottom line is that since
$$
 \rho(\u_1\u_1)  \rho(\u_2\u_2) \rho(\u_1\u_2)
 =\dist(z_1,z_2)    \dist(w_1,w_2) 
$$
and
$$
\widehat{\rho}(\u_1\u_1) \widehat{\rho}(\u_2\u_2) \widehat{\rho}(\u_1\u_2)\widehat{\rho}(\u_1\u_3)\widehat{\rho}(\u_2\u_3)  
=  \dist(w_1, z_\#) \dist(z_1, w_1) \dist(z_2, w_2)  .
$$
By formula \eqref{estimate_5}, this implies that
\begin{equation*}
\widehat{\rho}(\u_1\u_1) \widehat{\rho}(\u_2\u_2) \widehat{\rho}(\u_1\u_2)\widehat{\rho}(\u_1\u_3)\widehat{\rho}(\u_2\u_3)  
\ll \rho(\u_1\u_1)  \rho(\u_2\u_2) \rho(\u_1\u_2) \prod_{u_j \in \mathfrak{V}}\rho(\u_1\u_j)\rho(\u_2\u_j)  .
\end{equation*}
Combined with formula  \eqref{estimate_6}, this shows that 
$$
 \prod_{ \u_j\in \mathfrak{U}' }  \widehat{\rho}(\u_1\u_j) 
  \prod_{\begin{subarray}{c} \u_j\in \mathfrak{U}'\\
j >1\end{subarray}} \widehat{\rho}(\u_2\u_j) 
\ll \prod_{ \u_j\in \mathfrak{U}' }  \rho(\u_1\u_j) 
  \prod_{\begin{subarray}{c} \u_j\in \mathfrak{U}'\\
j >1\end{subarray}} \rho(\u_2\u_j) . 
$$
Hence, we can disregard the full contribution  of the vertices $\u_1, \u_2$.  
By induction, we can repeat this procedure until the set  $\mathfrak{A}^* = \emptyset$. In this case, by \eqref{pair_4}, we are left with the edge-set
\begin{equation} \label{pair_5}
E'=  \big(\mathfrak{A}\times \mathfrak{A}   \big) \cup \big(\mathfrak{A} \times  \mathfrak{V}  \big) \cup 
\big( \mathfrak{V} \times \mathfrak{V} \big) . 
\end{equation}

By definition, 
\[
\prod_{\u_i\u_j \in \mathfrak{A}\times \mathfrak{A}}\rho(\u_i\u_j)   = 1 , 
\]
and using the separation conditions \eqref{pair_2},  we  have
\[
\prod_{\u_i\u_j \in  \mathfrak{V} \times \mathfrak{V} } \rho(\u_i\u_j)  \gg 1 . 
\]
Thus, by formula \eqref{estimate_2}, we have proved that 
\begin{equation}\label{estimate_3}
\prod_{ i, j  \in [k+\ell]}\frac{ \widehat{\rho}(\u_i\u_j) }{\rho(\u_i\u_j) } 
\ll  \prod_{\u_i\u_j \in \mathfrak{A} \times \mathfrak{A}} \widehat{\rho}(\u_i\u_j) 
 \prod_{\u_i\u_j \in \mathfrak{A} \times \mathfrak{V}} \frac{ \widehat{\rho}(\u_i\u_j) }{\rho(\u_i\u_j) }  . 
\end{equation}

Let $\u_1 \in  \mathfrak{A}  $ and  $\u_2 \in  \mathfrak{V}  $. 
Since  $\dist(z_2, w_2) \ge \epsilon$ and $\dist(z_1, w_1) \le  \dist(z_1, z_2) \wedge  \dist(w_1, w_2)   $, we must have 
\[
\dist(z_1, z_2) \ge \epsilon/3 
  \hspace{.5cm}\text{or}  \hspace{.5cm}
\dist(w_1, w_2) \ge \epsilon/3 .
\] 
Moreover, also because of the separation condition \eqref{pair_2}, there is at most one point $\u_2 \in  \mathfrak{V}  $ such that one of this condition is not true. 
This implies that
\begin{align*}
 \prod_{ \u_j\in \mathfrak{V} }  \rho(\u_1\u_j) & \gg \dist(z_1, z_2) \wedge
\dist(w_1, w_2) \\
&\gg \dist(z_1, w_1) =  \widehat{\rho}(\u_1\u_1) .
\end{align*}
So, we have proved that for any vertex $\u_1 \in  \mathfrak{A}  $, we have
\begin{equation}\label{estimate_4}
 \widehat{\rho}(\u_1\u_1) 
 \prod_{\u_j \in \mathfrak{V}} \frac{ \widehat{\rho}(\u_1\u_j) }{\rho(\u_1\u_j) }
 \ll 1  .
\end{equation}
Combining the estimates \eqref{estimate_3} and \eqref{estimate_4}, this completes the proof of the lemma. \qed \\

\section{Upper bound}
\label{sec:ub}

In this section we prove the upper bound in Theorem~\ref{thm:main}.  
\begin{theorem} 
Assume that $V$ is real analytic and regular. Then, there exists a constant $C_V>0,$ which depends only on $V$, 
so that for any sequence $y_N \to \infty$ as $N\to\infty$, 
   \[ 
     \liminf_{N\to\infty}\Pr\left[\sup_{q\in \C} \EF(q) \le \log N + \tfrac{1}{2}\log\log N + C_Vy_N  \right]
 = 1.
  \]
  \label{prop:absub}
\end{theorem}

We begin by giving a completely deterministic relaxation of the problem.
Throughout this section, we will use 
\begin{equation} \label{tilde_g_function} 
      \widetilde{g}(x) := \int_\R \log|x- u|\,\LSD(du),
\end{equation} 
to denote the log-potential of the equilibrium measure.
\begin{lemma}
  For any $N\in\N$, the function $\EF(q)$ is subharmonic in $\C\backslash \mathbf{J}$ and, almost surely,
\[ 
  \sup_{q\in \C} \EF(q) = \max_{q\in {\mathbf{J}}}  \EF(q) \vee 0 .  
\]
\end{lemma}
\begin{proof}
  This fact that $\EF(q)$ is subharmonic follows directly from the fact that $\tilde{g}(x)$ is harmonic $\C\backslash \mathbf{J}$ and $\log |P_N(q)|$ is subharmonic on all $\C.$
 Hence the maximum of $\EF$ is attained either on $\mathbf{J}$ or along a sequence of points going to $\infty.$  However, we may write
    \[
	\EF(q) =
        \int_\R \log|1-uq^{-1}|\,\ESD(du)
        -N\int_\R \log|1-uq^{-1}|\,\LSD(du).
    \]
    Hence, $\EF(q) \to 0$ as $q \to \infty,$ and so the lemma follows.
\end{proof}

Next, on the account that $\EF$ arises as the modulus of a polynomial, we can reduce the task of bounding $\EF$ on $[-1,1]$ to bounding $\EF$ on a deterministic set of cardinality $O(N),$ by losing only an absolute constant. To do so, we may use the following reformulation of Lemma 4.3 in~\cite{CMN}. 

\begin{lemma}
  Fix $N \in \mathbb{N}.$
  Let $P_N$ be a polynomial of degree $N.$  Then
  \[
    \max_{x \in [-1,1]} |P_N(x)| 
    \leq 
    14 
    \max_{k \in [2N+1]} |P_N(x_k)| ,
  \]
  where $x_k= \cos( \pi (k-1)/2N)$ for $k \in [2N+1] := \{ 1, \dots , 2N+1\}$.   
 \label{lem:CMN}
\end{lemma}
\begin{proof}
The bound is an immediate consequence of \cite[Lemma 4.3]{CMN} which states that for any polynomial $W_N$ of degree at most $2N\in \N$,
\begin{equation*} \label{eq:CMN}
    \max_{|\omega|=1} |W_N(\omega)| 
    \leq 
    14 
    \max_{k \in [4N]} |W_N(e^{\pi i k/(2N)})|. 
\end{equation*}
If we let 
\[
    W_{2N}(\omega) = P_N(J(\omega))\omega^{N},
  \]
  then $W_{2N}$ is a polynomial of degree $2N$ such that for any $|\omega| = 1$, 
  \[
    |W_{2N}(\omega)| 
    =|P_N(J(\omega))|. 
  \]  
  This implies the claim, noting that as $J$ double-covers the unit interval it suffices to take $\omega \in \{e^{\pi i (k-1)/(2N)} : k \in [2N+1]\}.$ 
\end{proof}

This allows us to give a deterministic interpolation bound which applies to the
recentered log potential $\EF.$ 
\begin{corollary}   \label{cor:chbyshevbound}
  Let $V$ be real analytic.
  There is a constant $C_V$,  so that for any $N \in \N$ there is a deterministic set $\left\{ w_k \right\} \subset \overline{\mathbf{J}}$ of cardinality at most $C_V\!\cdot\!N$ so that
  for any polynomial $P_N$ of degree $N \ge m$, 
  \[
    \sup_{x\in \mathbf{J}}  \bigl\{ \log | P_N(x) | - N \tilde{g}(x)\bigr\}
    \le  \max_{k}  \bigl\{ \log | P_N(w_k) | - N \tilde{g}(w_k)\bigr\}
    +\log(42).
  \]
\end{corollary}

\begin{proof} 
  We can bound interval-by-interval of the support $\mathbf{J}$ (c.f.\ \eqref{J}), in that for each $0 \leq j \leq m,$ we show there is a constant $C_{V,j}$ and a set $\{ w_k^j \}\subset [b_j,a_{j+1}]$  so that
  \[
    \max_{x\in [b_j, a_{j+1}]}  \bigl\{ \log | P_N(x) | - N \tilde{g}(x)\bigr\}
    \le  \max_{k \in [C_{V,j}N+1]}  \bigl\{ \log | P_N(w_k^j) | - N \tilde{g}(w_k^j)\bigr\}
    +\log(42) .
  \]
The result then follows from taking the union of interpolating points $\cup_{j=1}^m \{ w_k^j \}$ and $C_V =1+ \sum_{j=0}^m C_{V,j}$. 
 So we fix $j$ and we may assume that $b_j = -1$ and $a_{j+1}=1$ by using an affine change of variables.
Note that  from the Euler--Lagrange equation \eqref{EL}), 
  \[
    \widetilde{g}(x) = V(x) + \ell_{V}
    \quad
    \text{for all}
    \quad 
    x \in [-1, 1].
  \]
We can also assume that $ \ell_{V}=0$ and it suffices to bound
\[
\max_{x\in[-1,1]}  \bigl\{ \log | P_N(x) | - N V(x)\bigr\}  = \log \max_{x\in[-1,1]} \bigl\{ \ |P_N(x)| e^{-NV(x)} \bigr\}  . 
\]
  As $V$ is real-analytic, $V(J(z))$ is analytic in the annulus $1-2\delta < |z| < 1+2\delta$ for a small $\delta>0$.  In particular, we may write for $\theta\in[0,\pi]$, 
 \[
  e^{-NV(\cos \theta)} = r_0 + \sum_{m=1}^\infty 2 r_m \cos(m \theta)
 \]
 where $\cos \theta = J(e^{i\theta})$ and
  \[
    r_m=
    \frac{1}{2\pi i} \oint_{|z|=1+\delta}e^{-N V(J(z)) } z^{-m-1} dz . 
  \]
  In particular, for any $K>0$, there exists $C_{V,K}$ so that $|r_m| \le e^{-KN-\delta m/2}$ if  $m \ge C_{V,K} N$. 
  This shows that for $\theta\in[0,\pi]$, 
 \[
\left|   e^{-NV(\cos \theta)} - r_0 - \sum_{m<C_{V,K} N} 2 r_m \cos(m \theta) \right| \le C_\delta e^{-K N} .
 \]
  Hence, letting $H_N(x) = r_0 + \sum_{m<C_{V,K} N} 2 r_m T_m(x)$ where $T_m$ are Chebyshev polynomials and choosing $K$ sufficiently large compared to $\max_{[-1,1]}|V|$, we obtain for $x\in[-1,1]$,
 \begin{equation} \label{approx}
  |1- H_N(x) e^{N V(x)}| \le 1/2 . 
\end{equation}
We conclude that $e^{-NV(x)} \le 2 H(x)$ and 
\[
   \max_{x\in[-1,1]}  \bigl\{ \log | P_N(x) | - N \tilde{g}(x)\bigr\}
    \le  \log \left( 2 \max_{x\in[-1,1]} \big\{ |P_N(x)| H_N(x) \big\}  \right)
\]
  Applying Lemma \ref{lem:CMN} (using the degree of $P_N H$ has degree  $C_{V,j}N$), we further obtain
  \[
  \max_{x\in [-1,1]}  \bigl\{ \log | P_N(x) | - N \tilde{g}(x)\bigr\}
    \leq\max_{k \in [2C_{V,j} N+1]}  \bigl\{ \log\big( |P_N(x_k)| H(x_k)\big) \bigr\} + \log(28).
  \]
   Using \eqref{approx} and the Euler--Lagrange equation again, $H(x_k) \le 3/2 e^{-N  \widetilde{g}(x_k) }$ for all $k$ so that 
\[
  \max_{x\in [-1,1]}  \bigl\{ \log | P_N(x) | - N \tilde{g}(x)\bigr\} 
  \le \max_{k \in [2C_{V,j} N+1]} \bigl\{ \log | P_N(x_k) | - N \tilde{g}(x_k)\bigr\} + \log(42) . 
\]
  Thus, going back to the original interval $[b_j,a_{j+1}]$ we conclude by choosing the mesh-points, 
  \begin{equation}\label{eq:explicits}
    w_k^j=\tfrac{a_{j+1}+b_j}{2} + \tfrac{a_{j+1}-b_j}{2}\cos\big(\tfrac{\pi(k-1)}{C_{V,j} N}\big)
    \quad
    \text{for}
    \quad 
    k \in [C_{V,j} N + 1]. \qedhere
  \end{equation}
 \end{proof}

 On the other hand, we claim that to control the maximum value of the field $\EF$, we can look slightly into the complex plane from $\mathbf{J}$. This will allow us to study more regularized statistics instead. 

\begin{proposition}
  Let $V$ be real analytic and regular.
  Let $C_V$ and $\left\{ w_k \right\}$ be as in Corollary \ref{cor:chbyshevbound}.
  Almost surely, for any $y \geq 0,$ 
\begin{equation} \label{max_inequality}
  \sup_{x\in \mathbf{J}} \EF(x) \le     \max_{\{w_k\}} \EF( w_k- i y/N ) + \log(42) + y \frac{ \pi \|\LSD\|_\infty}{3}.
\end{equation}
  \label{prop:regularization}
\end{proposition}
\begin{proof}
First, by monotonicity, for any $x\in\R$ and $y>0$, 
  \[
        \int_\R \log|x-u|\,\ESD(du) \leq
        \int_\R \log|x-i y -u |\,\ESD(du).
  \]
  On the other hand,
 \begin{align*}
    \int_\R \log|x-iy-v|\,\LSD(u)du
    &=
    \int_\R \log|x-v|\,\LSD(u)du
    +
    \int_\R \int_0^y \frac{t}{(x-u)^2+t^2}dt \LSD(u)du. \\
    \intertext{Thus, making a change of variables in the last integral and bounding uniformly the equilibrium density, we obtain}
    \int_\R \log|x-i y-u|\,\LSD(u)du
    &\leq
    \int_\R \log|x-u|\,\LSD(u)du
    +
   y \|\LSD\|_\infty \int_\R \int_0^1 \frac{t}{u^2+t^2}dtdu. \\
    &=
    \int_\R \log|x-u|\,\LSD(u)du
    +
   y \frac{ \pi \|\LSD\|_\infty}{3} .
  \end{align*}
  
  Combining both inequalities for the log potentials of $\ESD$ and $\LSD$,
   we conclude that for any $k\in[C_V \cdot N+1]$,
   $$
 \EF( w_k) \le \EF( w_k- i y/N )      +   y \frac{ \pi \|\LSD\|_\infty}{3} ,
    $$
and the  bound \eqref{max_inequality} follows directly from Corollary~\ref{cor:chbyshevbound}.
\end{proof}

The inequality \eqref{max_inequality} and the bound of Corollary~\ref{prop:expbound} for the exponential moments of the field $\EF$ allow us to easily complete the proof of theorem~\ref{prop:absub}.

\begin{proof}[Proof of theorem~\ref{prop:absub}]
Given any sequence $y_N \in[1,N]$,  we let $q_k = w_k- i y_N/N$ where $\{w_k\}$ is the mesh from Corollary \ref{cor:chbyshevbound}.
By a union bound and Markov's inequality, we have
\begin{equation*} 
  \Pr\left[   \max_{\{w_k\}} \EF( q_k )  \ge {\log N + \tfrac12 \log\log N}    \right]
\le  \frac{1}{N^2\log N} \sum_{\{w_k\}} \Exp\left[ e^{2\EF(q_k )} \right] . 
\end{equation*}
To estimate the RHS of this formula, we use the bound of Corollary~\ref{prop:expbound}.  Namely, we get 
\begin{equation}  \label{Markov_inequality}
  \Pr\left[   \max_{\{w_k\}} \EF( q_k )  \ge {\log N + \tfrac12 \log\log N}    \right]
  \le  C\sum_{{\{w_k\}}}  \frac{\mathrm{R}(q_k)^2 }{Ny_N \log N}.
\end{equation}

By  formula \eqref{R_function},  we have 
 \begin{equation} \label{R_inequality}
 \mathrm{R}(q)^2 =\prod_{j=0}^m |(q-a_j)(q-b_j)|^{-1/2} 
   \le C \max_{r \in \{a_j,b_j\}_{j=1}^m}  \frac{1}{\sqrt{|q-r|}} .
\end{equation}  
To estimate the sum on the RHS of \eqref{Markov_inequality}, we need the explicit form \eqref{eq:explicits} of the points $\{w_k\}.$  For every cut, there is a higher density of points from the family $\{q_k^j\}$ near the endpoints $\{a_j, b_j\}$  and for $k \in [C_{V,j} N + 1]$,
\[
\min_{r \in \{a_j,b_j\}} |q_k^j-r| \ge   c  \frac{y_N+(k-1)^2 \wedge (C_{V,j}-k-1)^2}{N^2} .
\]
This implies that $\sum_{{\{w_k\}}} \mathrm{R}(q_k)^2 =O(N\log N) $ accounting for all the cuts. 
  Returning to \eqref{Markov_inequality}, we conclude that
\begin{equation}  \label{Markov_inequality2}
  \Pr\left[   \max_{ k \in [C_VN+1] } \EF( q_k )  \ge {\log N + \tfrac12 \log\log N}    \right]
  \le \frac{C_V}{y_N},
\end{equation}
which implies Theorem~\ref{prop:absub} on account of Proposition \ref{prop:regularization}.
\end{proof}

\section{Lower bound proof}
\label{sec:lb}
This section is concerned with the proof of Theorem~\ref{thm:subharmoniclb}.  We will begin by giving an overview of the proof.  In what follows we fix $\delta >0$ a small positive constant, that we will ultimately take to $0.$  We let $n = \lceil \log N \rceil,$ and we define $n_0 = \lfloor (1-\delta) n \rfloor.$  

Recall that we would like to show, for some $\omega \in \T,$ the unit circle, that $\WEF(\omega \zeta_{n_0})$ is large.
Consider biasing the field $\WEF$ by the Radon-Nikodym derivative $e^{2\WEF(\omega\zeta_{n_0})}\cdot\Exp\left[e^{2\WEF(\omega\zeta_{n_0})}\right]^{-1}.$  If this were a Gaussian field, this bias would have the effect of changing the means of the field (c.f.\,Lemma~\ref{lem:means}).  In particular, along the ray $\left\{ \omega \zeta_j \right\}_1^\infty$ the mean of $\WEF(\omega \zeta_j)$ would then be roughly $j.$  The variance of $\WEF(\omega\zeta_{n_0})$, however, would be unchanged.  In particular, after biasing, it is typical behavior that $\WEF(\omega\zeta_{n_0}) \approx n_0.$ Moreover, on the event that $\WEF(\omega\zeta_{n_0}) \approx n_0,$ the Radon-Nikodym factor behaves like a constant, which would allow us to produce a lower bound for the event that $\WEF(\omega \zeta_{n_0})$ is large without the biasing factor.

Hence, this suggests one strategy for showing the field is large: computing a biased first and second moment of $\WEF(\omega \zeta_{n_0})$.  Since we do not have direct access to field moments, we approximate these moments by using linear combinations of exponentials.  We refer to these estimates as the \emph{field moment calculus} (see Section~\ref{sec:fmc}).

This by itself leads to lower bounds for quantities like $\Pr\left[ \WEF(\omega \zeta_n) > n \right].$
 To produce a lower bound for the maximum of $\WEF,$ we need more, as such events are too far from being independent.  Hence, we apply a modified second moment method, which appears frequently in the study of log-correlated fields.  Roughly speaking, we must constrain the trajectory of the walk $\WEF(\omega \zeta_j) \approx j$ for many values of $j$.  Moreover, such a constraint turns out to be typical for the direction $\omega$ along which the maximum is achieved, and hence this condition does not reduce the probability too much.

Let $\eta=\eta_N$ be a slowly growing sequence to be specified later.  Let $b_k = k\lfloor n_0/ \eta \rfloor$ for all $k=0, \dots \eta$. 
Define $r=r(\delta,N,\eta) \in \N$ to be the smallest integer so that $z \in \D$ having $\dH(0,z) > b_{r}$ have $|z| > 1-N^{-2\delta}.$
For $\omega \in \T,$ define the subset of fields $\BL$ by
\begin{equation} \label{good_event}
    \BL[\omega] = \{\F : |\F(\omega\zeta_{b_k}) - \F(i\zeta_{b_r}) - (b_{k}-b_r)| \leq \eta \sqrt{n},~\forall~ r < k \leq \eta \}. %
\end{equation}
The point $\zeta_{b*}$ from the introduction is $\zeta_{b_r}.$

Recall the set $\Omega = \left\{ e^{i(\tfrac{\pi}{2} + h e^{-n_0})} : h \in \mathbb{Z}, |h| < N^{-\delta}e^{n_0} \right\}.$
Roughly, we would like to apply the second moment method to the counting function 
\[
    \tilde{Z} =  \sum_{\omega \in \Omega}
    \one[{\WEF \in \BL[\omega]}].%
\]
Then, we would estimate
\(
  \Pr\left[ \tilde{Z} \geq 1 \right]
  \geq \frac{
    (\Exp[ \tilde{Z} ])^2
  }{
    \Exp[ \tilde{Z}^2 ]
  }.
\)
One of the subtleties of this strategy is that to get this probability going to $1,$ one basically needs that for most pairs $(\omega_1,\omega_2),$
\begin{align*}
  &\Pr[\WEF \in \BL[\omega_1] \cap \BL[\omega_2]] 
  =\Pr[\WEF \in \BL[\omega_1] ]%
  \cdot\Pr[\WEF \in \BL[\omega_2] ](1+o(1)).
\end{align*}
In the case of branching random walk, this is achieved using actual independence of the two events, once a suitably small common ancestral tree is discarded.  In our case, no such independence is available.  %
 Further, the need to constantly bias and unbias the measure by exponential factors which are not totally determined by being on the event $\BL$ causes losses which are not $(1+o(1)).$ 

We solve this problem by changing the second moment formalism. 
Define for any $\omega \in \Omega$
\begin{equation} \label{rv_Y}
    \MFI(\omega) = e^{2\WEF(\omega \zeta_{n_0}) - 2\WEF(i \zeta_{b_{r}})} \one[ \WEF \in \BL].
\end{equation}
In terms of this replacement for the indicator, we form the biased counting variable
\begin{equation*}
  Z=  \sum_{\omega \in \Omega} \MFI(\omega). 
\end{equation*}
The advantage of using $\MFI$ is that a nearly sharp (up to $1+o(1)$ multiplicative error) estimate for $\Exp\left[ \MFI(\omega_{1})\MFI(\omega_{2}) \right]$ when $\omega_1$ and $\omega_2$ are well-separated is attained by simply dropping the indicator.

We again want to show that this variable is non-negative, as this implies one of the indicated events holds.  Applying Cauchy-Schwartz inequality,
\begin{equation}
  \label{eq:2mm}
  \Exp[ \one[Z > 0]]
  \geq \frac
  { \bigl(\sum_{\omega}\Exp\left[ \MFI(\omega) \right]\bigr)^2}
  {\sum_{\omega_1,\omega_2}\Exp\left[\MFI(\omega_1)\MFI(\omega_2)\right]},
\end{equation}
where in both summations, $\omega$ ranges over $\Omega.$ 

As mentioned, an efficient upper bound for the denominator is simply to drop the indicators and use an estimate for mixed exponential moments.  The numerator, however, is less amenable to a direct estimate but the moral here is that with enough uniformity in the estimates, mixed exponential moments are enough to estimate the RHS of \eqref{eq:2mm} from below.

\subsection{Field moment calculus}
\label{sec:fmc}

In what follows, we will fix an $\omega \in \Omega.$
First, for any finite bias term $\Bias,$ and any nonnegative measurable function $\phi$, we denote
\[
    \FMC{\Bias}[\phi(\WEF)] = \frac{\Exp \left[e^{\Bias{\WEF}}\phi(\WEF)\right]}
    {\Exp \left[e^{\Bias{\WEF}}\right]}.
\]
We will drop the dependence of the notation on $\Bias$ when 
\[
    \Bias{\F} =
    \Bias{\F}[\omega]:=
    2\F(\omega \zeta_{n_0})
    -2\F( i\zeta_{b_{r}}),
\]
which plays a special role.
This section is concerned with using the exponential moment criteria Assumption~\ref{ass:wusa} to produce estimates on biased moments of $\WEF.$ 
   
Hence, we recall the effect of biasing a jointly Gaussian vector by a linear functional of that vector is to change the mean of the Gaussian vector.  
\begin{lemma}
  Let $\mathbf{z}$ and $\mathbf{y}$
  be finite subsets of $\D$ 
  and let $F\in C(\D)$ be a real valued function on $\D$. 
  Let 
  \begin{align*}
    \Bias: \F  &\mapsto 
    \sum_{z \in \mathbf{z}} 2\F(z)
    -\sum_{y \in \mathbf{y}} 2\F(y),
  \end{align*}
  and let $\mu: \zeta \mapsto \Exp\left[ \WGF(\zeta) \Bias(\WGF) \right].$
  Then
  \[
    \frac{\Exp\left[ F(\WGF)e^{\Bias{\WGF}} \right]}{\Exp\left[e^{\Bias{\WGF}} \right]}
    =
    \Exp\left[F(\WGF+\mu)\right].
  \]
  \label{lem:means}
\end{lemma}
\begin{proof}
  The field $\WGF$ is almost surely continuous (in fact harmonic) in $\D,$ and so it suffices to assume that $F$ is a function which depends on the value of $\WGF$ at finitely many points of $\D$ by a density argument.  Let $\mathbf{w}$ be a finite set of points containing $\mathbf{z}$ and $\mathbf{y}.$
    The vector $\left( \WGF(w) \right)_{w \in \mathbf{w}}$ is jointly Gaussian.  
    Let $\Sigma$ be the covariance matrix of this vector.  We can find a vector $v \in \R^k$ representing $\Bias$ in that
    \[
        \Exp\left[ F(\WGF)e^{\Bias{\WGF}} \right]
        =
        \int_{\R^k}
        \frac{F(x)e^{v^tx}}{\sqrt{(2\pi)^k |\det \Sigma|}}
        e^{\tfrac{-x^t \Sigma^{-1} x}{2}}\,dx.
    \]
    We can now set $u = \Sigma v$ and complete the square, to get
    \[
        \frac{
            \Exp\left[ F(\WGF)e^{\Bias{\WGF}} \right]
        }
        {
            e^{\tfrac{u^t \Sigma^{-1}u}{2}}
        }
        =
        \int_{\R^k}
        \frac{F(x)}{\sqrt{(2\pi)^k |\det \Sigma|}}
        e^{\tfrac{-(x-u)^t \Sigma^{-1} (x-u)}{2}}\,dx.
    \]
    The constant $e^{\tfrac{u^t \Sigma^{-1}u}{2}} = \Exp\left[e^{\Bias{\WGF}} \right],$ as can be verified by setting $F \equiv 1$ and changing variables in the integral.  Moreover, the right hand side is exactly the claimed expression in the lemma.
\end{proof}

Recall that $\mathfrak{W}_{ \ell,\delta}(\Bias_\omega)$ is the set of biases of the form 
\[
  \F \mapsto
  \Bias_\omega(\F)  
  + \sum_{z \in \mathbf{z}} 2\F(z)
  - \sum_{y \in \mathbf{y}} 2\F(y)
\]
for some $\mathbf{z},\mathbf{y} \subset \mathcal{D}_{N,\delta}$ of cardinalities $\ell$ having the property that $\mathbf{z}$ and $\mathbf{y}$ are paired: there exists a bijection $\phi : \mathbf{z} \to \mathbf{y}$ so that \eqref{condition_2} holds.  We let $\mathfrak{W}_{\ell,\delta}^*(\Bias_\omega) \subset \mathfrak{W}_{\ell ,\delta}(\Bias_\omega)$ be the biases such that in addition to \eqref{condition_2}, for all $z \in \mathbf{z},$
\[
  \dH(z,\phi(z)) \leq 1.
\]

This simple change of mean lemma will be used in the following form.
\begin{corollary}
  Fix $\ell \in \mathbb{N}.$
  Suppose Assumption $\WUSA$ (\ref{ass:wusa}) holds.  
  Let $\mu(\cdot) = \Exp\left[ \WGF(\cdot)\Bias[\omega](\WGF) \right]$ be the mean of $\WGF$ under the bias $\Bias_\omega.$ For any $\delta > 0$ there is a constant $C=C(\ell,\delta)$ so that for all $\omega \in \Omega$ and all $\Bias \in \mathfrak{W}^*_{\ell,\delta}(\Bias_\omega)$
  \[
   \left|
    \FMC{\Bias_\omega}
    \left[ 
      e^{\Bias(\WEF)-\Bias_\omega(\WEF)}
    \right]
    -
    \Exp \left[ 
      e^{\Bias(\mu+\WGF)-\Bias_\omega(\mu+\WGF)}
    \right]
    \right|
    \leq CN^{-\delta}.
  \]
  \label{cor:emc}
\end{corollary}
\begin{proof}
  By definition of $\FMC_{\Bias_\omega}$ and an application of Assumption~\ref{ass:wusa},
  \begin{align*}
    \FMC{\Bias_\omega}
    \left[ 
      e^{\Bias(\WEF)-\Bias_\omega(\WEF)}
    \right]
    &=\frac{\Exp[ e^{\Bias(\WEF)}] }{\Exp[ e^{\Bias_\omega(\WEF)}] } \\
    &=\frac{\Exp[ e^{\Bias(\WGF)}](1+O(N^{-\delta})) }{\Exp[ e^{\Bias_\omega(\WGF)}](1+O(N^{-\delta}))} \\
    &=
    \Exp \left[ 
      e^{\Bias(\mu+\WGF)-\Bias_\omega(\mu+\WGF)}
    \right]
    (1+O(N^{-\delta}))
  \end{align*}
  where, at last, we used lemma~\ref{lem:means} applied to the field $\WGF$. 
  So, to complete the proof, we just need to establish that 
  \[
  \frac{\Exp[ e^{\Bias(\WGF)}]}{\Exp[ e^{\Bias_\omega(\WGF)}]} \asymp_\ell 1.
  \]
  As $\WGF$ is Gaussian, this is equivalent to showing that
  \[
    |\Var( \Bias(\WGF)) - \Var(\Bias_\omega(\WGF))| \ll_\ell 1.
  \]
  Expanding the variances, we have
  \begin{align*}
    \Var( \Bias(\WGF)) - \Var(\Bias_\omega(\WGF))
    =
    \sum_{z \in \mathbf{z}} {4}&\Exp\left[ \Bias_\omega(\WGF)(\WGF(z)-\WGF(\phi(z)))\right] \\
    +
    \sum_{z,w \in \mathbf{z}} {4}&\Exp\left[ (\WGF(z)-\WGF(\phi(z)))(\WGF(w)-\WGF(\phi(w))) \right].
  \end{align*}
  Using Property (c) of Definition \ref{def:brwlike} and the fact that   $\dH(z,\phi(z)) \leq 1$, each summand in the above equation is bounded by an absolute constant, from which the desired conclusion on the variances follows.
\end{proof}

\begin{lemma}
 Fix $\ell \in \mathbb{N}$ and $\omega \in \Omega.$  
 Suppose Assumption $\WUSA$ (\ref{ass:wusa}) holds.
 Let $\Bias = \Bias_\omega.$
 There is a constant $C=C(\ell,\delta)$ so that for all $0 \leq \Delta \leq 1$ the following holds. Let $\left\{ z_i \right\}_1^\ell$ and $\left\{ y_i \right\}_1^\ell$ be points in $\mathcal{D}_{N,\delta}$ with $\dH(z_i,y_i) \leq \Delta$ for all $1 \leq i \leq \ell.$  Define 
  \[
    \Bias[i] : \F \mapsto 2\F(z_i) - 2\F(y_i).
  \]
  Let $\mu(\cdot) = \Exp\left[ \WGF(\cdot)\Bias{\WGF} \right]$ be the mean of $\WGF$ under the bias $\Bias$.   Then
  \[
    \left|
    \FMC{\Bias}
    \left[ 
      \textstyle{\prod_{i=1}^\ell} \Bias[i](\WEF)
    \right]
    -
    \Exp \left[ 
      \textstyle{\prod_{i=1}^\ell} \Bias[i](\mu + \WGF)
    \right]
    \right|
    \leq C(\Delta^{\ell+1} \vee N^{-\delta/4}).
  \]
  \label{lem:uberharmonic}
\end{lemma}
\begin{remark}
  For the field $\ZF,$
  in light of Proposition~\ref{prop:ZG}, this lemma also holds if $\Bias_\omega$ is replaced by some $\Bias \in \mathfrak{S}_{j,\epsilon,\delta}$ for some $\epsilon$ and $j.$  The proof is completely general and needs no alteration for this case.  Further, the errors are uniform in $\mathfrak{S}_{j,\epsilon,\delta}$ for fixed choices of $j$ and $\epsilon.$
\end{remark}
\begin{proof}

  We will write the proof for a general $\Bias \in \mathfrak{S}_{j,\epsilon,\delta}$ for some $j \in \N_0$ and some $\epsilon > 0.$
  Suppose that 
  \[
    \Bias(\F) = 
    \sum_{z \in \mathbf{z}} 2\F(z)
    -\sum_{y \in \mathbf{y}} 2\F(y) ,
  \]
 where $\mathbf{z}$ and $\mathbf{y}$ are finite subsets of $\mathcal{D}_{N,\delta}.$

  By the symmetry of the problem, it is enough to establish the one-sided bound:
  \begin{equation}
    \FMC{\Bias}
    \left[ 
      \textstyle{\prod_{i=1}^\ell} \Bias[i](\WEF)
    \right]
    \leq
    \Exp \left[ 
      \textstyle{\prod_{i=1}^\ell} \Bias[i](\mu + \WGF)
    \right]
    +C(\Delta^{\ell+1} \vee N^{-\delta/4}).
    \label{eq:uh1}
  \end{equation}
  We will approximate the increments $\Bias[i]$ by exponentials, in order to apply Assumption~\ref{ass:wusa}.  As we do not suppose anything on the locations of $\left\{ z_i \right\}_1^\ell$ or $\left\{ y_i \right\}_1^\ell,$ we use harmonicity in a nontrivial way to reduce the problem to one in which pairs are well-separated, and we give this argument first.

For each $i = 1,2,\dots,\ell$ we will define a contour $\gamma_i.$   Let $\mathbf{z}' = \mathbf{z} \cup \left\{ z_j \right\}_1^\ell$ and likewise for $\mathbf{y}'.$ We define
  \[
    \gamma_i = \partial\left(\cup_{w \in \mathbf{z}' \cup \mathbf{y}'} \overline{\HB(w,3i)}\right),
  \]
  that is $\gamma_i$ traces the boundary of the union of these balls with a positive orientation.  With this definition $\gamma_0$ is the set $\mathbf{z}' \cup \mathbf{y}'.$  As $\gamma_i$ is a piecewise smooth arc, we have there is a probability measure $\nu_{i}$ supported on $\gamma_i$ so that for any harmonic function $\varphi$ on $\mathcal{D}_{N,\delta/2}$ 
  \[
    \varphi(z_i) = \int_{\gamma_i} \varphi(w) \nu_{i}(dw).
  \]
 Let $M_i$ be a disk automorphism taking $z_i$ to $y_i.$  This map is analytic on $\D$ and has the property that $\dH(w, M_i(w)) = \dH(z_i,y_i)$ for all $w \in \D.$  In particular, $\WGF - \WGF \circ M_i$ is harmonic in $\mathcal{D}_{N,\delta/2}$ for all $N$ sufficiently large.  This allows us to give the representation,
  \[
    \WEF(z_i) - \WEF(y_i) = \int_{\gamma_i} (\WEF(w)-\WEF(M_i(w)))\nu_i(dw).
  \]
  
  Therefore, if we define the biases $\Bias[{(i,w)}]: \F \mapsto 2\F(w) - 2\F(M_i(w)),$ we can write
  \begin{equation}
    \FMC{\Bias}
    \left[ 
      \textstyle{\prod_{i=1}^\ell} \Bias[i](\WEF)
    \right]
    =\int_{\gamma_1}\dots\int_{\gamma_\ell}
    \FMC{\Bias}
    \left[ 
      \textstyle{\prod_{i=1}^\ell} \Bias[(i,w_i)](\WEF)
    \right]
    \nu_i(dw_i).
    \label{eq:uh2}
  \end{equation}
  Commuting the contour integration and the $\FMC$-expectation is easily justified using the inequality
  \begin{align*}
    |\textstyle{\prod_{i=1}^\ell} \Bias[(i,w_i)](\WEF)|
    &\leq \textstyle{\sum_{i=1}^\ell}  |\Bias[(i,w_i)](\WEF)|^\ell \\
    &\ll_\ell \textstyle{\sum_{i=1}^\ell} \cosh(\Bias[(i,w_i)](\WEF)),
  \end{align*}
  and Assumption~\ref{ass:wusa}, once we verify the provisions of that assumption.  %

  In that direction, we claim that for any subset $S \subset [\ell],$ and any points $\left\{ w_i \right\}_{i \in S}$ in the respective supports of $\{\nu_i\},$
  \begin{equation}
    \Bias + \sum_{i \in S} \pm \Bias[(i,w_i)] \in
    \mathfrak{W}^*_{|S|,\delta/2}(\Bias) \subseteq
    \mathfrak{W}^*_{\ell,\delta/2}(\Bias).
    \label{eq:uh3}
  \end{equation}
  By how $\{\gamma_i\}$ are chosen, we have that for any $w_i \in \gamma_i$ and $w_j \in \gamma_j$ with $i \neq j,$
  \[
    3 \leq \dH(w_i, w_j).
  \]
  As $\dH(w_i, M_i(w_i))\leq 1$ for all $i$ we also have that
  \[
    2 \leq \dH(M_i(w_i),w_j), 
    2 \leq \dH(w_i,M_j(w_j)) 
    \text{ and }
    1 \leq \dH(M_i(w_i),M_j(w_j)).
  \]
  This implies that 
  \(
    \Bias + \sum_{i \in S} \Bias[(i,w_i)] \in \mathfrak{W}^*_{\ell,\delta/2}
  \)
  as desired. %

The bottom line is that  if we establish \eqref{eq:uh1} in the case that the pairs $(y_i,z_i)$ are each at least distance $2$ to any other point of $\mathbf{z}' \cup \mathbf{y}',$ then using \eqref{eq:uh2} we obtain 
  \begin{align*}
    \FMC{\Bias}
    \left[ 
      \textstyle{\prod_{i=1}^\ell} \Bias[i](\WEF)
    \right]
    &=\int_{\gamma_1}\dots\int_{\gamma_\ell}
    \FMC{\Bias}
    \left[ 
      \textstyle{\prod_{i=1}^\ell} \Bias[(i,w_i)](\WEF)
    \right]
    \nu_i(dw_i) \\
    &\leq\int_{\gamma_1}\dots\int_{\gamma_\ell}
    \Exp \left[ 
      \textstyle{\prod_{i=1}^\ell} \Bias[(i,w_i)](\mu + \WGF)
    \right]
    \nu_i(dw_i)
    +C(\Delta^{\ell+1} \vee N^{-\delta/2}).
    \intertext{Moreover, by the almost sure harmonicity of $\mu$ and $\WGF,$ we conclude:}
    \FMC{\Bias}
    \left[ 
      \textstyle{\prod_{i=1}^\ell} \Bias[i](\WEF)
    \right]
    &\leq
    \Exp \left[ 
      \textstyle{\prod_{i=1}^\ell} \Bias[i](\mu + \WGF)
    \right]
    +C(\Delta^{\ell+1} \vee N^{-\delta/2}),
  \end{align*}
  as desired.

  Therefore, we have reduced the problem to showing \eqref{eq:uh1} in the case that \eqref{eq:uh3} holds.  To do this, we will replace the moments we wish to calculate by exponentials, to which our assumption applies.  More specifically, we will show that
  \begin{equation}
    \FMC{\Bias}
    \left[ 
      \textstyle{\prod_{i=1}^\ell} \Bias[i](\WEF)
    \right]
    -
    \FMC{\Bias}
    \left[ 
      \textstyle{\prod_{i=1}^\ell} (e^{\Bias[i](\WEF)} - 1)
    \right]
    \ll_{k,\ell,\delta}
    \Delta^{ \ell+1} \vee N^{-\delta/4}.
    \label{eq:uh4}
  \end{equation}
  Having performed this bound, we will be in a position to apply Assumption~\ref{ass:wusa} to the $\FMC$-expectation of exponentials.  In particular, a direct application of Corollary~\ref{cor:emc} shows that  
  \begin{equation}
    \FMC{\Bias}
    \left[ 
      \textstyle{\prod_{i=1}^\ell} (e^{\Bias[i](\WEF)} - 1)
    \right]
    -
    \Exp
    \left[ 
      \textstyle{\prod_{i=1}^\ell} (e^{\Bias[i](\mu + \WGF)} - 1)
    \right]
    \ll_{k,\ell,\delta} N^{-\delta/2} .
    \label{eq:guh}
  \end{equation}
  Therefore it will only remain to prove that for the Gaussian field $\WGF$, we have
  \begin{equation}
    \Exp
    \left[ 
      \textstyle{\prod_{i=1}^\ell} (e^{\Bias[i](\mu +\WGF)} - 1)
    \right]
    -
    \Exp
    \left[ 
      \textstyle{\prod_{i=1}^\ell} \Bias[i](\mu + \WGF)
    \right]
    \ll_{\ell}
    \Delta^{\ell+1}
    \label{eq:uh45}
  \end{equation}
  to complete the proof. Indeed, observe that summing \eqref{eq:uh4}, \eqref{eq:guh} and \eqref{eq:uh45} give \eqref{eq:uh1}.  %

  We begin by showing that \eqref{eq:uh4} holds.
It suffices to show that, for any $1 \leq p \leq \ell$ 
  \begin{equation}
    \FMC{\Bias}
    \biggl\{
    \left[ 
      \textstyle{\prod_{i=1}^p} (e^{\Bias[i](\WEF)} - 1 - \Bias[i](\WEF))
    \right]
    \cdot
    \left| 
      \textstyle{\prod_{i=p+1}^\ell} \Bias[i](\WEF)
    \right|
  \biggr\}
    \ll_{k,\ell,\delta}
    \Delta^{\ell+p} \vee N^{-\delta/4}.
    \label{eq:uh5}
  \end{equation}
  Since $e^{x}-1-x \geq 0$ for all $x \in \R$, the bound \eqref{eq:uh4} follows from this by adding and subtracting $\Bias_i(\WEF)$ inside the second product of \eqref{eq:uh4} and expanding.  The terms that result have the form of \eqref{eq:uh5} up to permutation of the indices of the products.  
  
  To establish \eqref{eq:uh5}, we start by applying Cauchy-Schwarz in the following way:
  \begin{align*}
    \biggl(&\FMC{\Bias}
    \biggl\{
    \left[ 
      \textstyle{\prod_{i=1}^p} (e^{\Bias[i](\WEF)} - 1 - \Bias[i](\WEF))
    \right]
    \cdot
    \left| 
      \textstyle{\prod_{i=p+1}^\ell} \Bias[i](\WEF)
    \right|
    \biggr\}
    \biggr)^2 \\
    \leq
    &\FMC{\Bias}
    \left[ 
      \textstyle{\prod_{i=1}^p} (e^{\Bias[i](\WEF)} - 1 - \Bias[i](\WEF))
    \right] \\
    \cdot&\FMC{\Bias}
    \biggl\{
    \left[ 
      \textstyle{\prod_{i=1}^p} (e^{\Bias[i](\WEF)} - 1 - \Bias[i](\WEF))
    \right]
    \cdot
    \textstyle{\prod_{i=p+1}^\ell} (\Bias[i](\WEF))^2 \biggr\}.
  \end{align*}
  We further bound this expression using the inequalities $e^{x}-1-x \leq 2(\cosh(x) - 1)$ and $x^2 \leq 2(\cosh(x) - 1).$  We simplify notation by writing $\varphi(x) = 2(\cosh(x) - 1).$  We then apply these bounds by writing
  \begin{align*}
    \biggl(&\FMC{\Bias}
    \biggl\{
    \left[ 
      \textstyle{\prod_{i=1}^p} (e^{\Bias[i](\WEF)} - 1 - \Bias[i](\WEF))
    \right]
    \cdot
    \left| 
      \textstyle{\prod_{i=p+1}^\ell} \Bias[i](\WEF)
    \right|
    \biggr\}
    \biggr)^2 \\
    \leq
    &\FMC{\Bias}
    \left[ 
      \textstyle{\prod_{i=1}^p} \varphi(\Bias[i](\WEF))
    \right] 
    \cdot\FMC{\Bias}
    \left[ 
      \textstyle{\prod_{i=1}^\ell} \varphi(\Bias[i](\WEF))
    \right]
    .
  \end{align*}

  Each of the expectations in this bound can be expanded into expectations of sums of exponentials of biases of the same form as in \eqref{eq:uh3}. Hence, by Corollary~\ref{cor:emc}, we can replace $\WEF$ by $\mu + \WGF$ incurring only an additive $N^{-\delta/2}$ error.  Moreover, by elementary manipulations, we will establish that for any $1 \leq p \leq \ell$
  \begin{equation}
    \Exp
    \left[ 
      \textstyle{\prod_{i=1}^p} \varphi(\Bias[i](\mu + \WGF))
    \right]
    \ll_{\ell} \Delta^{2p}
    \label{eq:uh6}
  \end{equation}
  Combining this with the previous displayed equation, we would have that
  \begin{align*}
    \biggl(&\FMC{\Bias}
    \biggl\{
    \left[ 
      \textstyle{\prod_{i=1}^p} (e^{\Bias[i](\WEF)} - 1 - \Bias[i](\WEF))
    \right]
    \cdot
    \left| 
      \textstyle{\prod_{i=p+1}^\ell} \Bias[i](\WEF)
    \right|
    \biggr\}
    \biggr)^2 \\
    \ll_\ell 
    &( \Delta^{2p} + N^{-\delta/2})
    ( \Delta^{2\ell} + N^{-\delta/2}) \\
    \ll_{\ell} & \Delta^{2p+2\ell} \vee N^{-\delta/2}.
  \end{align*}
  Taking square-roots, the estimates \eqref{eq:uh5} and \eqref{eq:uh4} follow provided that we prove \eqref{eq:uh6}. To do so, we begin by using the arithmetic-geometric mean inequality:
  \[
    \Exp
    \left[ 
      \textstyle{\prod_{i=1}^p} \varphi(\Bias[i](\mu + \WGF))
    \right]
    \leq
    \frac{1}{p}
    \sum_{i=1}^p
    \Exp
    \left[ \varphi(\Bias[i](\mu + \WGF))^p
    \right].
  \]
  By Definition~\ref{def:brwlike}, the variable $\Bias[i](\mu + \WGF)$ are Gaussian with mean $\mu(z_i)-\mu(y_i) \ll_k  \dH(z_i,y_i) \leq \Delta$ and variance $\Var(\WGF(z_i) - \WGF(y_i)) \ll \dH(z_i,y_i)^2 \leq \Delta^2.$
   Hence the variable $\Bias[i](\mu + \WGF)/\Delta$ satisfies a uniform Gaussian tail bound with constants depending only on $k$,

  The function $\varphi(x)$ vanishes quadratically near $0,$ so that $\varphi(x)x^{-2}$ can be bounded by $K\cosh(x)$ for some sufficiently large constant $K >0.$  In particular, we have  
  \[
    \varphi(x) \leq \Delta^2 \psi(x/\Delta)
  \]
  where $\psi(x) = K x^2 \cosh(x).$ Hence
  \[
    \Exp
    \left[ \varphi(\Bias[i](\mu + \WGF))^p
    \right] 
    \leq \Delta^{2p}
    \Exp
    \left[ \psi(\Bias[i](\mu + \WGF)/\Delta)^p
    \right] \ll_{k,\ell} \Delta^{2p}. 
  \]
  which completes the proof of \eqref{eq:uh6}.  

It just remains to prove the estimate \eqref{eq:uh45}.  However, the proof here is nearly identical to \eqref{eq:uh4}, so we just sketch it.  Analogously to \eqref{eq:uh4}, we can start by adding and subtracting $\Bias_i(\mu + \WGF)$ inside the first product and expanding.  This reduces the problem to bounding
    \begin{equation}
      \Exp
    \biggl\{
    \left[ 
      \textstyle{\prod_{i=1}^p} (e^{\Bias[i](\mu+\WGF)} - 1 - \Bias[i](\mu+\WGF))
    \right]
    \cdot
    \left| 
      \textstyle{\prod_{i=p+1}^\ell} \Bias[i](\mu + \WGF)
    \right|
  \biggr\}
    \ll_{k,\ell,\delta}
    \Delta^{\ell+1}.
    \label{eq:uh7}
  \end{equation}
  The exact same bounds used in reducing \eqref{eq:uh5} to \eqref{eq:uh6} can be applied to reduce \eqref{eq:uh7} to \eqref{eq:uh6}, which completes the proof.
\end{proof}

By piecing together various local perturbations, it is therefore possible to remove the requirement that $z_i$ and $y_i$ be close in the previous lemma, at the expense of a larger error term.  
\begin{proposition}
  Fix $\ell \in \mathbb{N},$ $0 < \alpha < 1$ and $\omega \in \Omega,$ and let $\Bias=\Bias_\omega.$ 
 Suppose Assumption $\WUSA$ (\ref{ass:wusa}) holds.
  There is a constant $C=C(\ell,\alpha,\delta)$ so that the following holds. Let $\left\{ z_i \right\}_1^\ell$ and $\left\{ y_i \right\}_1^\ell$ be points in $\mathcal{D}_{N,\delta}.$ 
  Define 
  \[
    \Bias[i] : \F \mapsto 2\F(z_i) - 2\F(y_i).
  \]
  Let $\mu(\cdot) = \Exp\left[ \WGF(\cdot)\Bias{\WGF} \right]$ be the mean of $\WGF$ under the bias $\Bias.$  Then
  \[
    \left|
    \FMC{\Bias}
    \left[ 
      \textstyle{\prod_{i=1}^\ell} \Bias[i](\WEF)
    \right]
    -
    \Exp \left[ 
      \textstyle{\prod_{i=1}^\ell} \Bias[i](\mu + \WGF)
    \right]
    \right|
    \leq Ce^{-(\log N)^\alpha}.
  \]
  \label{prop:mixedmoments}
\end{proposition}
\begin{proof}
  For each $1 \leq i \leq \ell,$
  let $\gamma_i$ be the geodesic from $y_i$ to $z_i$ parameterized by hyperbolic arc length.  This geodesic necessarily lies in the wedge 
  \[
  \left\{ 
    re^{i(\theta+\tfrac\pi 2)}
    \ |\ 0 \leq r \leq 1 - N^{-1+\delta},
    |\theta| \leq N^{-\delta}
  \right\},
  \]
  which is geodesically convex (compare with the definition of $\mathcal{D}_{N,\delta}$ \eqref{eq:domain}, for which the radius $r \ge 1 - N^{-\delta} $ as well). 
  It may be necessary to deform $\gamma_i$ to stay within $\mathcal{D}_{N,\delta},$ and it can be performed by replacing any segment in 
  \[
  \left\{ 
    re^{i(\theta+\tfrac\pi 2)}
    \ |\ 0 \leq r \leq 1 - N^{-\delta},
    |\theta| \leq N^{-\delta}
  \right\},
  \]
  by a segment that travels along  the curve $r= 1 - N^{-\delta}$,   $|\theta| \leq N^{-\delta}$. This inner curved boundary only has length $O(1),$ as its angular length is $O(N^{-\delta})$ and for $z$ on this arc, the hyperbolic metric is the euclidean one scaled by $(1-|z|)^{-1} = N^\delta$.  Hence, all curves that arise this way have length $O(\log N),$ as $\mathcal{D}_{N,\delta}$ is contained in a hyperbolic ball of radius $\log N.$

  Let 
   \(
   0 = 
   t^{(i)}_0
   \leq
   t^{(i)}_1
   \leq
   t^{(i)}_2
   \leq 
   \cdots
   \leq 
   t^{(i)}_{p_i}
   \)
  be evenly spaced points, with spacing $e^{-(\log N)^{\alpha}},$ save for possibly the last, which may be shorter.  For any $1 \leq j \leq p_i$ define
  \(
  \Bias[(i,j)] : \F \mapsto 
  2\F(\gamma_i(t^{(i)}_{j}))
  -2\F(\gamma_i(t^{(i)}_{j-1})).
  \)
  Then we can trivially bound
  \begin{align*}
    &|\FMC{\Bias}
    \left[ 
      \textstyle{\prod_{i=1}^\ell} \Bias[i](\WEF)
    \right]
    -
    \Exp \left[ 
      \textstyle{\prod_{i=1}^\ell} \Bias[i](\mu + \WGF)
    \right]| \\
    \leq
    \sum_{j_1 =1}^{p_1}
    \cdots
    \sum_{j_\ell =1}^{p_\ell}
    &
    |\FMC{\Bias}
    \left[ 
      \textstyle{\prod_{i=1}^\ell} \Bias[(i,j_i)](\WEF)
    \right]
    -
    \Exp \left[ 
      \textstyle{\prod_{i=1}^\ell} \Bias[(i,j_i)](\mu + \WGF)
    \right]
    |.
  \end{align*}
  To this difference, we now apply Lemma~\ref{lem:uberharmonic}.  Thus the summands can be uniformly controlled by $O(e^{-(\ell+1)(\log N)^{\alpha}}).$  The number of summands, meanwhile is at most $O( e^{ \ell (\log N)^{\alpha} + \ell \log\log N}).$ 
\end{proof}

\subsection{Estimating $\Exp[ \one[Z > 0]]$}
\label{sec:2p1}

Using these field moments, we can estimate the conditional probability of $\BL$, \eqref{good_event}.  For the remainder of this section, we suppose Assumption $\WUSA[2]$ holds (\ref{ass:wusa}). 
\begin{lemma}
  For any $\epsilon > 0,$ by making $N$ sufficiently large, we have that uniformly in $\omega \in \Omega,$
  \[
    \Exp[e^{\Bias[\omega](\WGF)}](1-\epsilon) \leq \Exp[\MFI(\omega)] \leq \Exp[e^{\Bias[\omega](\WGF)}](1+\epsilon).
  \]
  \label{lem:1pLB}
\end{lemma}
\begin{proof}
  The upper bound proceeds by a trivial bound and a direct application of Assumption~\ref{ass:wusa}:
  \[
    \Exp[\MFI(\omega)] 
    \leq 
    \Exp[e^{\Bias[\omega](\WEF)}] 
    =
    \Exp[e^{\Bias[\omega](\WGF)}]
    (1+O(N^{-\delta})).
  \]
  For the lower bound, by \eqref{rv_Y},  it suffices that show that,    when  the parameter $N$ is sufficiently large, $ \FMC[\one[{\WEF \in \BL[\omega]}]] \ge 1-\epsilon$, where 
 $ \FMC =  \FMC_{\Bias[\omega]}$.
  By a union bound, we simply estimate
  \begin{align*}
       \FMC[\one[\WEF \notin \BL]]
    &\leq
    \sum_{k=r+1}^{\eta}
    \FMC\left[ \one[ 
      |\WEF(\omega\zeta_{b_k}) - \WEF(i\zeta_{b_r}) - (b_{k}-b_r)| > \eta \cdot n^{1/2}
    ]\right] \\
    &\leq
    \sum_{k=r}^{\eta}
    \frac
    {\FMC[ (\WEF(\omega\zeta_{b_k}) - \WEF(i\zeta_{b_r})  - (b_{k}-b_r))^2 ] }
    {\eta^2 n}. \\
    \intertext{By Proposition~\ref{prop:mixedmoments}, we can compute the first and second moments of $\WEF(\omega\zeta_{b_k}) - \WEF(i\zeta_{b_r})$ under $\FMC.$  Using this Gaussian comparison, the mean and variance of this increment are $b_k - b_r$ up to a uniformly bounded $O(1)$ error.  
    Therefore, this $\FMC$-expectation is bounded above by $n + O(1),$ which leads to
    }
    \FMC[\one[\WEF \notin \BL]]
    &\ll
    \sum_{k=r+1}^{\eta}
    \frac
    {n}
    {\eta^2 n} \leq \eta^{-1}.
  \end{align*}
  This concludes the lower bound as $\eta \to \infty.$
\end{proof}

We now turn to estimating $\Exp[\MFI(\omega_1)\MFI(\omega_2)]$ for various values of $(\omega_1,\omega_2).$  There will be two regimes of $|\omega_1-\omega_2|$ in which we make different estimates.  We introduce the midpoint $\HM=\HM(\omega_1,\omega_2),$ defined as the integer closest to $-\log |\omega_1-\omega_2|$ that is also at least $n_0.$  This is roughly the height at which the $\omega_1$ and $\omega_2$ rays branch.  In the first regime, where $\HM < b_r,$ the rays have branched early enough that there is essentially no correlation between $\MFI(\omega_1)$ and $\MFI(\omega_2).$  Otherwise, we must appropriately take advantage of the barrier information in $\MFI(\omega)$ to assure the correlation is not too high.

The estimate for small $\HM$ is no more complicated than the estimates in Lemma~\ref{lem:1pLB}.
\begin{lemma}
  For all $\epsilon > 0$ and all $\omega_1,\omega_2 \in \Omega$ with $\HM(\omega_1,\omega_2) \leq (1-\epsilon) b_r$, if $N$ is  sufficiently large, then
  \[
    \Exp[\MFI(\omega_1)\MFI(\omega_2)]
    \leq
    \Exp[\MFI(\omega_1)]\Exp[\MFI(\omega_2)](1+\epsilon).
  \]
  \label{lem:finefield2pUB1}
\end{lemma}
\begin{proof}
  We estimate the left hand side by the trivial bound 
  \[
    \Exp[\MFI(\omega_1)\MFI(\omega_2)]
    \leq \Exp[e^{\Bias[\omega_1](\WEF)+\Bias[\omega_2](\WEF)}].
  \]
  This bias $\Bias[\omega_1] + \Bias[\omega_2]$ must have all points separated by a distance independent of $n$ to apply Assumption~\ref{ass:wusa}.  
  
  For a hyperbolic triangle with side lengths $a,b,c$ with $\theta$ the angle opposite $a,$ the hyperbolic law of cosines says that
  \[
    \cosh a = 
    \frac{\cosh(b+c)}{2}(1-\cos\theta)
    +\frac{\cosh(b-c)}{2}(1+\cos\theta).
  \]
  Hence, we can derive a formula for the $\cosh$ of the distance between $\xi_1\omega_1$ and $\xi_2 \omega_2$ 
  for any $\xi_{i} \in \left\{ \zeta_{b_r}, \zeta_{n_0} \right\}$ for $i=1,2.$  Moreover, as $N^{-\delta} > |\theta| \gg e^{ -(1-\epsilon)\cdot b_r},$ we have in all cases that 
  \[
    \cosh( \dH(\xi_1\omega_1,\xi_2 \omega_2)) \gg e^{2\epsilon b_r} \to \infty.
  \]

  Hence we have that
  \[
    \Exp[\MFI(\omega_1)\MFI(\omega_2)]
    \leq \Exp[e^{\Bias[\omega_1](\WGF)+\Bias[\omega_2](\WGF)}](1+O(N^{-\delta})).
  \]
  As this is a Gaussian expectation, it suffices to compute covariances of the two biases to show the expectation splits: specifically, the covariance $\Exp\Bias[\omega_1](\WGF)\Bias[\omega_2](\WGF).$   Using part (c) of Definition~\ref{def:brwlike} and the assumption on $\HM(\omega_1,\omega_2)$ in the statement of the lemma, we have that
  \[
    \Exp \WGF(\xi_1 \omega_1)\WGF(\xi_2 \omega_2)
    = -\frac{1}{2}\log|\sin\left(\tfrac{\arg(\omega_1\omega_2^{-1})}{2}\right)|
    +K(\arg\omega_1,\arg\omega_2) + o(1)
  \]
  for any $\xi_{i} \in \left\{ \zeta_{b_r}, \zeta_{n_0} \right\}$ for $i=1,2.$
  Moreover, the error $o(1)$ is uniform in $\omega_{i},$ and so we get that
  \[
    \Exp[e^{\Bias[\omega_1](\WGF)+\Bias[\omega_2](\WGF)}]
    =
    \Exp[e^{\Bias[\omega_1](\WGF)}]\Exp[e^{\Bias[\omega_2](\WGF)}](1+o(1)).
  \]
  Hence, applying Lemma~\ref{lem:1pLB}, we can conclude that
  \[
    \Exp[\MFI(\omega_1)\MFI(\omega_2)]
    \leq
    \Exp[\MFI(\omega_1)]\Exp[\MFI(\omega_2)](1+o(1))
  \]
  uniformly in $(\omega_1,\omega_2)$ satisfying the hypotheses of the lemma.
\end{proof}

This lemma covers all but a vanishing fraction of pairs $(\omega_1,\omega_2)$ we need to consider.  However,  we must also assure that the remaining terms are not too correlated.  This is the content of the following lemma.
\begin{lemma}
  There is a constant $C>0$ so that for all $\omega_1,\omega_2 \in \Omega$ with $\HM \geq \frac12 b_r$ and all $N$ sufficiently large
  \[
    \Exp[\MFI(\omega_1)\MFI(\omega_2)]
    \leq 
    C\Exp[\MFI(\omega_1)]\Exp[\MFI(\omega_2)]
    e^{\HM-b_r + {n}{\eta}^{-1} + \eta n^{1/2}}.
  \]
  \label{lem:finefield2pUB2}
\end{lemma}
\begin{proof}
  Unlike when $\HM$ was small, in the setting of the previous lemma, $\Bias[\omega_1]+\Bias[\omega_2],$ which is the biasing term that appears in the left-hand side, will have much too large a variance.  This is because, by analogy with branching random walk, this bias counts the segment before the $\omega_1$ and $\omega_2$ rays split twice.  Hence, we would like to re-bias the exponential weight on the left-hand side.  Ideally we would replace $\omega_1\zeta_{b_r}$ with $\omega_1\zeta_{\HM}$ in the biasing term.  However, we do not have exact control on the value of $\Bias[\omega_1](\omega_1\zeta_{\HM}),$ and so we instead choose an approximation over which we do.  To this end, let $b_* \in \left\{ b_r,b_{r+1},\dots,b_\eta \right\}$ be the closest element to $\HM$ that is larger than or equal to $\HM,$
  so that \(b_* - \HM \leq \tfrac{n}{\eta}.\)

  Let 
\begin{align*}
    \Bias[*] : \F \mapsto\
   & \Bias{\F}[\omega_1]+\Bias{\F}[\omega_2]
    +2\F(i \zeta_{b_r})
    -2\F(\omega_1 \zeta_{b_*}) \\
    &=    \Bias{\F}[\omega_2]+  2\F(\omega_1 \zeta_{n_0}) -2\F(\omega_1 \zeta_{b_*}). 
  \end{align*}
  Recall that
  \[
    \BL[\omega] = \{ |\F(\omega\zeta_{b_k}) - \F(i\zeta_{b_r}) - (b_{k}-b_r)| \leq \eta \cdot n^{1/2},~\forall~ r < k \leq \eta \}. 
  \]
  Hence, by \eqref{rv_Y}, 
  \[
    \Exp[\MFI(\omega_1)\MFI(\omega_2)]
    \leq
    \Exp\left[ 
      e^{\Bias{\WEF}[*] + 2(b_* - b_r) +  2\eta n^{1/2}}
    \right].
  \]
  To this exponential moment, we may apply Assumption~\ref{ass:wusa}, as we have a uniform lower bound on the distance between all points.  
  
  Therefore, to complete the evaluation of this exponential moment, we just need to estimate the variance of $\Bias{\WGF}[*].$ By how $b_*$ was chosen (part (d) of definition~\ref{def:brwlike}), we have
  \[
    \Exp\left[ \Bias{\WGF}[\omega_2]( \WGF(\omega_1\zeta_{n_0}) - \WGF(\omega_1 \zeta_{b_*})) \right] = O(1),
  \]
  as all but the $O(1)$ terms cancel from the covariance.
  Hence, we conclude that
  \[
    \tfrac 12 \Var\left( \Bias{\WGF}[*] \right) = (n_0 - b_r) + (n_0 - b_*) + O(1),
  \]
  which leads to the conclusion
\begin{align} \notag
    \Exp[\MFI(\omega_1)\MFI(\omega_2)]
    &\ll
       e^{
      2n_0 + b_* - 3b_r +2 \eta n^{1/2}  
    }  \\
    & \label{lwb_estimate}\ll
    e^{
      2n_0 + \HM - 3b_r + \eta n^{1/2} + n \eta^{-1} 
    }.
  \end{align}
  Finally, by Lemma~\ref{lem:1pLB}, for $i=1,2$
  \[
    \Exp[\MFI(\omega_i)]
    \asymp \Exp\left[ e^{\Bias_{\omega_i}(\WGF)} \right]
    =e^{n_0 - b_r + O(1)},
  \]
  which completes the proof.
 \end{proof}

 We are now able to show the desired lower bound, i.e.\,the proof of Theorem~\ref{thm:subharmoniclb}.
\begin{proof}[Proof of Theorem~\ref{thm:subharmoniclb}]
  Recall that we let
  \[
  \Omega = \left\{ e^{i(\tfrac{\pi}{2} + h e^{-n_0})} : h \in \mathbb{Z}, |h| < N^{-\delta}e^{n_0} \right\},
  \quad
  \text{and}
  \quad
    Z=  \sum_{\omega \in \Omega} \MFI(\omega) . 
  \]
  We bound the probability that $Z > 0$ from below using that
  \begin{align}
   \Exp[ \one[Z > 0]]
    \geq \frac
    { \bigl(\sum_{\omega}\Exp\left[ \MFI(\omega) \right]\bigr)^2}
    {\sum_{\omega_1,\omega_2}\Exp\left[\MFI(\omega_1)\MFI(\omega_2)\right]},
    \label{eq:csagain}
  \end{align}
  where $\omega_1,\omega_2$ run over the set $\Omega.$
  We now partition the sum in the denominator according to the value of $\HM(\omega_1,\omega_2).$
  Specifically, we let 
  $I_1$ be the sum
  \[
    I_1 = 
    \sum_{\substack{\omega_1,\omega_2 \\ \HM(\omega_1,\omega_2) \leq \tfrac34
      b_r
    }}
    \Exp\left[\MFI(\omega_1)\MFI(\omega_2)\right],
  \]
  and we let $I_2$ be the sum over the remaining pairs $(\omega_1,\omega_2).$

  By Lemma~\ref{lem:finefield2pUB1}, for all $N$ sufficiently large, we have the simple bound
  \begin{align}
    I_1 
    &\leq 
    \sum_{\substack{\omega_1,\omega_2 \\ \HM(\omega_1,\omega_2) \leq \tfrac34
      b_r
    }}
    (1+\delta)\Exp\left[\MFI(\omega_1)\right]\Exp\left[\MFI(\omega_2)\right], 
    \nonumber \\
    &\leq 
    (1+\delta)
    \bigl(\sum_{\substack{\omega}} \Exp\left[\MFI(\omega)\right]\bigr)^2.
    \label{eq:I1bound}
  \end{align}

  To control $I_2,$ we partition the sum according to the size of $\HM(\omega_1,\omega_2).$
  For a given $\omega$ and an integer $p,$ there are at most $O(e^{n_0-\ell})$ many $\omega' \in \Omega$ so that $\HM(\omega,\omega') \geq \ell.$ 
By the estimate \eqref{lwb_estimate}, this implies that
    \begin{align*}
    I_2=&\sum_{3b_r/4 \le \ell \le n_0}
    \sum_{\substack{\omega_1,\omega_2 \in \Omega \\ \HM(\omega_1,\omega_2) = \ell}}
    \Exp\left[\MFI(\omega_1)\MFI(\omega_2)\right]\\
    \ll& n_0
    |\Omega| e^{3n_0-3b_r + {n}{\eta}^{-1} + \eta n^{1/2}} \\
    \ll&
    N^{-\delta}e^{4n_0-3b_r + {n}{\eta}^{-1} + \eta n^{1/2} + \log n}.
  \end{align*}

  For comparison, we need an estimate for the numerator in \eqref{eq:csagain}.  By Lemma~\ref{lem:1pLB}, we have 
  \begin{align*}
    \bigl(\sum_{\substack{\omega}} \Exp\left[\MFI(\omega)\right]\bigr)^2
    &\gg |\Omega|^2e^{2n_0 - 2b_r} \\
    &\gg N^{-2\delta}e^{4n_0 - 2b_r}.
  \end{align*}

  By how $b_r$ was chosen, we have $b_r = 2\delta \log N + O(1).$  Recalling that $n = \log N + O(1),$ we conclude that 
  \[
    I_2 \ll N^{-\delta/2} 
    \bigl(\sum_{\substack{\omega}} \Exp\left[\MFI(\omega)\right]\bigr)^2.
  \]
  Thus, applying this and \eqref{eq:I1bound} to \eqref{eq:csagain}, we get
  \[
   \Exp[ \one[Z > 0]]
   \geq 
   \frac{1}{1+\delta + o(1)}.
  \]
  This completes the proof.
\end{proof}

\printbibliography

\end{document}